\documentclass{article}

\usepackage{amssymb}
\usepackage{a4}
\usepackage{QED}
\usepackage{url}
\usepackage[dvipsnames]{xcolor}

\usepackage[latin2]{inputenc}
\usepackage[T1]{fontenc}
\usepackage[french, magyar, english]{babel}
\usepackage{hyperref}
\usepackage{mathrsfs}

\makeindex
   
\newtheorem{theorem}{Theorem}[section]

\newtheorem{corollary}[theorem]{Corollary}

\newtheorem{definition}[theorem]{Definition}

\newtheorem{lemma}[theorem]{Lemma}

\newtheorem{remark}[theorem]{Remark}

\newcommand{\m}{\mu}

\newcommand{\lmp}{\l \mu \mu'}

\def\N{\ifmmode{\rm I\mkern-3.1mu
N\mkern0.5mu}\else{\rm I\kern-.16em
N\hskip0.5pt\ }\fi\relax} 

\def\tre{\sqsubseteq}
\def\a{\alpha}
\def\b{\beta}

\def\e{\varepsilon}
\def\g{\gamma}
\def\l{\lambda}

\def\r{\rho}

\def\th{\theta}
\def\lm{\l \mu}
\def\lmp{\l \mu \mu'}

\def\ra{\rightarrow}

\def\si{\sigma}
\def\r{\rho}

\def\two{\twoheadrightarrow}

\def\v{\vdash}
\def\G{\Gamma}

\def\F{\displaystyle\frac}
\def\bdf{\begin{definition}}
\def\edf{\end{definition}}
\def\be{\begin{enumerate}}
\def\ee{\end{enumerate}}
\def\bd{\begin{description}}
\def\ed{\end{description}}
\def\ite{\begin{itemize}}
\def\ete{\end{itemize}}
\def\bp{\begin{proof}}
\def\ep{\end{proof}}
\def\bl{\begin{lemma}}
\def\el{\end{lemma}}
\def\bt{\begin{theorem}}
\def\et{\end{theorem}}
\def\bpt{\begin{prooftree}}
\def\ept{\end{prooftree}}
\def\bdf{\begin{definition}}
\def\edf{\end{definition}}

\def\<{\langle}
\def\>{\rangle}
\def\s{\sigma}

\def\f{\rightarrow}
\def\B{\bot\!\!\!\!\bot}

\title{Normalization properties of $\l\m$-calculus using realizability semantics }

\author{P\'eter Batty\'anyi\thanks{Department of Computer Science, Faculty of Informatics, University of Debrecen, Kassai \'ut 26, 4028 Debrecen, Hungary,
{\tt battyanyi.peter@inf.unideb.hu}} $\;\;\;$
        and$\;\;\;$
Karim Nour\thanks{Universit\'e Savoie Mont Blanc, CNRS, LAMA, LIMD,
73000 Chamb\'ery, France,
{\tt karim.nour@univ-smb.fr}}\\\\}

\begin{document}
\maketitle

\begin{abstract}
In this paper, we present a general realizability semantics for the simply typed $\l\m$-calculus. Then, based on this semantics, we derive both weak and strong normalization results for two versions of the $\l\m$-calculus equipped with specific simplification rules.  The novelty in our method, in addition to its more systematic approach, lies in its applicability to a broader set of reduction rules without relying on the usual postponement technique. Our approach is original in that it introduces a parameter into the definition of the model, thus establishing a general result which we can then apply to systems with different sets of reduction rules by adjusting the parameter accordingly. Our saturation conditions also lead to a neat characterization of typable $\lambda\mu$-terms.\\
\end{abstract}


{\bf Subject code classifications :} 03B40 ; 03B70 ; 03F05 ; 68Q42.

\section{Introduction}

\subsection*{The Curry-Howard correspondence} 

The Curry-Howard correspondence establishes a fundamental bridge between formal proofs in intuitionistic logic and
computer programs, particularly through formalisms such as the lambda calculus or combinatory logic. It can be viewed from
two complementary perspectives : on the one hand, as a process of extracting programs from proofs; on the other hand, as an
assignment of types-interpreted as logical formulas-to classes of programs. In both cases, reductions carried out on proofs
correspond precisely to computational reductions on the associated programs.
This deep connection between logic and computation has drawn the attention of many researchers concerned with ensuring
the correctness of programs by extracting them directly from mathematical proofs whose validity can be formally verified. In
well-designed systems, where the number of proof rules is finite, verifying a proof is often simpler than verifying a program
built independently.

Among the major contributions to this field, one may cite the work of Girard, Krivine, Leivant, and Parigot, who relied on
second-order logical systems associated with lambda calculus programs \cite{Gir,Kri2,Kri3,Par1,Par2,Par3}. These systems offer a dual advantage: a relatively
simple syntax and sufficient expressiveness to formalize proofs of complex mathematical results. In particular, they allow for
an elegant definition of standard algebraic data types (such as booleans, integers, lists, or trees), while ensuring uniqueness of
representation: for each object of a given type, there exists a unique (non-executable) normal lambda-term whose type
corresponds exactly to that object. This enables a natural representation of data within the calculus. It is also worth
noting that, on the computational side, the simplification of lambda terms is not limited to the well-known beta reduction. Using these additional rules, it can be ensured, for example, that the unique representative of a proof that an object $n$ is an integer is precisely the Church numeral $\underline{n}$.

Krivine's programming theorem illustrates this correspondence perfectly \cite{Kri1}. It states that any proof of the totality of a function defined by induction on a data type can be translated into a correct program, expressed in lambda calculus, that effectively
computes that function. Moreover, Krivine showed that such proofs can always be formulated within the framework of
intuitionistic logic, without resorting to reductio ad absurdum, thus reinforcing the constructive nature of the approach.

\subsection*{Parigot's $\lambda \mu$-calculus}

Several researchers have explored the extension of the Curry-Howard correspondence to classical logic. Griffin was the first
to open this path, observing that the natural encoding corresponding to classical reasoning (via reductio ad absurdum) closely
resembles the call/cc (call-with-current-continuation) instruction from the Scheme programming language \cite{Gri}. This instruction captures the current execution stack, enabling control over program flow, for instance, to throw an exception or change the control structure.

Building on this idea, several logical systems along with corresponding calculi have been proposed to provide a computational interpretation of classical logic. Among the earliest and most notable are: Rehof and S\o rensen's $\lambda\Delta$-calculus \cite{Reh}, Krivine's $\lambda_C$-calculus \cite{Kri3}, Parigot's $\lambda\mu$-calculus \cite{Par1}, Barbanera and Berardi's $\lambda_{Sym}$-calculus and $\lambda_{\mathscr{C}\tau^-}$-calculus \cite{BerBar, BB2}, and finally Curien and Herbelin's $\lambda\mu\tilde{\mu}$-calculus \cite{Cur-Her}. Each of these systems has its own motivations, strengths, and limitations.
Other systems have been introduced, based on a more refined analysis of logical systems, particularly linear logic and the notion of polarization. Among the notable works, those by Wadler \cite{Wad03}, Laurent \cite{Lau02}, Zeilberger \cite{Zei09}, and Terui \cite{Ter08} have significantly enriched this approach, each contributing to a deeper understanding of computational duality, polarization, and complexity within the framework of focused calculi. Munch-Maccagnoni's $L$-calculus \cite{MM} provides, in particular, an elegant framework that unifies computation and proof dynamics through precise control of logical phases.

In the present work, we choose to build upon Parigot's $\lambda\mu$-calculus, for several reasons. First, it constitutes a very
natural extension of earlier work in second-order intuitionistic logic. Parigot's system is based on natural deduction, with
standard introduction and elimination rules, and two additional rules to handle negation and reductio ad absurdum. On the
computational side, the fundamental lambda calculus constructors (application and abstraction) are retained, with three major
additions: a new class of variables, called classical variables, used to encode assumptions handled via reductio ad absurdum;
two new syntactic constructs- brackets (for eliminating intuitionistic negation) and a $\mu$ operator (for eliminating
classical negation, i.e., reductio ad absurdum). Parigot also imposes certain syntactic restrictions on his logical system to
limit the variety of outcoming programs and to facilitate the proof of key properties. One such restriction requires that reductio ad
absurdum only be used after an explicit elimination of negation.

On the computational side, a new reduction rule, denoted $\mu$, is added to the standard beta reduction. This rule
corresponds to the elimination of a cut linked to reductio ad absurdum. From the outset, Parigot also introduced two
administrative reductions, $\rho$ and $\delta$, used to simplify programs by removing unnecessary $\mu$-constructs arising
from redundant classical reasoning.

Parigot's system satisfies all the expected properties: confluence of the calculus, type preservation under reduction, and strong
normalization in the typed setting. However, from his earliest work, Parigot identified a major difficulty in representing data.
For example, a classical proof that an object $n$ is an integer may yield a normal $\lambda\mu$-term (normal under all
reductions) which differs from the Church numeral $n$. Parigot referred to such objects as classical integers. 
The problem is that the output of a program manipulating classical integers is not easily interpretable: it becomes difficult to determine the value of the computed integer. This problem is not merely related to the encoding of data in the calculus. Even when the calculus is extended with constants representing data types, a proof of the totality of a function may yield a non-identifiable object, that is, a term whose actual value cannot be determined. This occurs because the proof of totality may rely on proof by contradiction, which leads to the construction of terms whose components arise directly from non-constructive reasoning. Such operators, although valid within the calculus, lack explicit computational content.

To address this, Parigot proposed several solutions: an external algorithm for extracting the value of a classical integer from
its representation; the use of storage operators introduced by Krivine, which simulate call-by-value in a call-by-name
language \cite{Kri3}; and the introduction of a new reduction rule, $\mu'$, which enables the retrieval of the value of a classical integer, provided it is applied to the reduced form of the term representing that integer \cite{Par2}.

\subsection*{The $\mu'$-rule}

The original version of Parigot's $\lambda\mu$-calculus has some significant limitations. In particular, due to the syntactic
restrictions it imposes, it is not possible to encode a proof of the classical tautology $\neg\neg A \rightarrow A$ without
introducing free variables of type $\neg\bot$. This issue arises for other classical tautologies as well. Another major limitation
concerns B{\"o}hm's theorem and separation property, which no longer holds in the presence of many reduction rules. In his dissertation, Py constructed two distinct closed normal terms, $M$ and $N$, such that there exists no term $L$ with $(L)M$ reducing to the first projection $\lambda x.\lambda y.x$ and $(L)N$ to the second projection $\lambda x.\lambda y.y$ \cite{Py,Dav-Py}.

To overcome these difficulties, de Groote proposed a more flexible syntactic variant of the $\lambda\mu$-calculus \cite{de Gro4}. In particular, he removed the constraint linking $\mu$-abstraction and application, splitting the rules associated with negation:
one for introduction and another for elimination, without further restrictions. This modification allowed him to construct an
abstract machine for the $\lambda\mu$-calculus. Furthermore, Saurin showed that this variant satisfies the separation
property \cite{Sau1, Sau2}. Moreover, this version allows for very simple proofs of classical tautologies without resorting to free variables. For these reasons, we have chosen to adopt the version of the $\lambda\mu$-calculus proposed by de Groote.

Let us now consider the $\mu'$ reduction rule, which is the symmetric counterpart of the $\mu$ rule and whose introduction
is quite natural. While $\mu$ allows the propagation of an argument inside a term, $\mu'$ enables the propagation of a
function. However, adding this rule breaks confluence: a term may reduce to distinct outcomes.
Nevertheless, the $\mu'$ rule preserves types in the simply typed system (without quantifiers). In contrast, in second-order
systems, ensuring type preservation requires assigning algorithmic content to quantifier rules.
Py studied such a system, but only in a call-by-value setting. For this typed version, he proved confluence and strong
normalization. Additionally, Parigot had posed a long-standing open question: is the system with both $\mu$ and $\mu'$
strongly normalizing? This question was first answered positively by David and Nour \cite{Dav-Nou4}, followed by a simpler proof proposed in \cite{BN2}.The same authors also established strong normalization results for the simply typed system with rules $\mu$ and $\mu'$ together with the $\beta$ rule.

In our work, we choose to investigate de Groote's flexible variant of the $\lambda\mu$-calculus, particularly to incorporate
the $\mu'$ rule into Parigot's reduction system, which is essential for retrieving the value of a classical integer from its
representation. However, this syntactic flexibility permits problematic derivations, such as a sequence of two reductio ad
absurdum steps, equivalent to a proof of $\bot$ by contradiction. To remedy this, Py introduced a new reduction rule, denoted
$\varepsilon$, designed to eliminate such undesirable constructions.
Surprisingly, Batty\'anyi discovered in his thesis that certain well-typed terms fail to be strongly normalizing in the presence of
the simplification rules $\rho$ or $\varepsilon$. This unexpected result led us to focus on weak normalization for the simply
typed system. We also verified that this system preserves the uniqueness of data representation: each integer has a unique
representation \cite{BN3}. At this stage, we have formalized the proof only for natural numbers, but it naturally extends to other data types.

In this article, we demonstrate that de Groote's $\lambda\mu$-calculus, extended with the rules $\beta$, $\mu$, $\mu'$, $\rho$, $\delta$, and $\varepsilon$, satisfies the weak normalization  property in the simply typed setting. Although these results were already established in our previous work \cite{BN3}, the semantic approach adopted here provides a uniform framework for obtaining the expected properties of both calculi: $\beta\mu\rho\delta\varepsilon$ and $\beta\mu\mu'\rho\delta\varepsilon$. 
We emphasize that, due to the presence of the $\mu'$ rule, the latter calculus is not strongly normalizing. Thus, only weak normalization results can be expected. In contrast to \cite{BN3}, the normalization proofs presented here are relatively short once suitable reducibility candidates are identified.

However, a direct extension of these results to a second-order system does not appear feasible. On the one hand, the $\mu'$ rule fails to preserve types in such a framework; on the other, our proofs rely on arithmetical arguments that are intrinsically tied to the simply typed setting. It is worth noting that Raffalli (as reported in Py's thesis) provided an example of a term of type $A$ that, via the $\mu'$ rule, can be reduced to an arbitrary term of type $B$. To avoid such anomalies, Py proposes assigning algorithmic content to the rules governing the universal quantifier. While promising, this approach significantly increases the complexity of the system and lies beyond the scope of the present work.

This paper has two main objectives. First, we show, using a single realizability semantics, that it is possible to obtain proofs of either strong or weak normalization, depending on the chosen set of reduction rules for the calculus. Second, we hope that this kind of semantics can eventually be extended to a type system based on second-order classical logic (with quantifiers), despite the necessary modifications to the calculus and the introduction of additional reduction rules


\subsection*{Normalization Properties}

As previously mentioned, the authors of this paper, in collaboration with David, have established in a series of works several normalization results for the simply typed version of the $\l\m$-calculus. These include strong normalization in the absence of the $\m'$ rule, as well as weak normalization when this rule is present. The corresponding proofs are technically intricate, relying on subtle inductions over $n$-tuples involving term size, type size, and bounds on strongly normalizing terms. To obtain these results, it is often necessary to treat the reduction rules separately: the $\b$ rule on one side, and the $\m$, $\m'$ and simplification rules on the other. This separation is supported by technical commutation lemmas. Such a differentiated treatment of the rules already appears in Parigot's work, where the proofs primarily focus on the $\b$ and $\m$ rules.

While the syntactic and arithmetical approach offers valuable insight into the mechanisms underlying (strong or weak) termination, it also has significant limitations. The introduction of new reduction rules typically requires reworking the entire proof structure to accommodate new behaviors, sometimes leading to major surprises, such as the loss of strong normalization. Furthermore, this type of proof does not readily extend to type systems with quantifiers. These considerations motivate the exploration of semantic approaches based on suitably adapted reducibility candidates.

Among the well-established techniques for proving the normalization of typed $\l$-calculi, realizability holds a central place. Originally introduced by Tait \cite{Tai} to prove the normalization of System $T$, and later extended by Girard \cite{Gir0}  to System $F$, this method, also known as normalization by reducibility or normalization by logical relations, interprets types as sets of terms, called realizers, whose construction ensures desirable properties such as termination. Over the years, many refinements and generalizations of this method have been developed; a detailed survey can be found in Gallier's work \cite{Gal}.

Beyond its applications to intuitionistic systems, realizability techniques have been successfully adapted to classical logic  \cite{BerBar,Par4}. A major contribution in this direction is due to Krivine \cite{Kri6}, who introduced a framework for classical realizability based on a $\l$-calculus with control, interpreted through an abstract machine. In this setting, terms are evaluated against stacks, and control (which embodies classical reasoning) is achieved through mechanisms that allow stacks to be saved and restored. This formalism has proved remarkably fruitful: on the logical side, it has led to the construction of new models of set theory and notably extended Cohen's forcing technique \cite{Kri4,Kri5,Kri6,Kri7,Kri8}; on the computational side, it has offered alternative tools for analyzing the computational content of classical proofs.

Let us begin by recalling that the notion of reducibility candidates, originally introduced in the context of the $\l$-calculus and later extended to its various refinements, can serve two distinct purposes: on the one hand, to establish normalization results, as we discussed earlier; on the other hand, to describe the behavior of typed terms, as in \cite{Kri2,Kri3,NS}. Another line of work \cite{Nou3} has focused on the realizability semantics of the simply typed $\l\m$-calculus, aiming to construct a general semantic framework with a view toward a completeness result, that is, an equivalence between typability and type inhabitation. This investigation reveals several challenges, most notably the integration of simplification rules into the semantic setting. The result obtained, although partial, is not fully satisfactory, as it relies on a modification of the semantics that lacks naturalness and whose interpretation remains somewhat unintuitive. This work thus highlights significant obstacles to defining a truly general realizability semantics for the simply typed $\l\m$-calculus. In \cite{Dav-Nou2}, it was shown that techniques based on reducibility candidates fail to establish strong normalization results in the presence of the $\mu'$ rule. The standard technical lemmas usually required in this context do not hold, which highlights the specific nature of the $\lambda\mu$-calculus when this rule is included.

Our approach, adopted in this paper, consists in defining a semantics based on reducibility candidates, with the goal of establishing normalization properties. This construction is also designed to allow, via this semantics, the analysis of the behavior of certain typed terms, although this aspect is not developed in the present work. A distinctive feature of our method lies in the fact that all reduction rules are considered simultaneously, without distinguishing between logical rules and simplification rules. The construction of the semantics is parameterized by a specific set, denoted $\B$, which is then used to establish various normalization results, regardless of whether the $\mu'$ rule is present or not.

In normalization proofs for other systems, this set is often fixed from the outset, for instance by taking the set of strongly normalizing terms. In contrast, our approach isolates the essential properties that this set must satisfy in order to ensure a correctness result. This allows us to specify it as needed, and thus to obtain several results within the same semantic framework.

Certain technical difficulties led us to restrict our study to the set of typable terms. This is not a limitation, since the properties we aim to establish concern only such terms. Moreover, we have shown that the set of typable terms is the only one satisfying our saturation conditions. It follows that, in order to establish a property on typable $\lambda\mu$ terms, it suffices to show that the set of terms satisfying this property forms a saturated set.

\subsection*{Summary of the Paper}

The paper is structured as follows: the next section is devoted to the relevant notions and already known results concerning the simply-typed $\l\m$-calculus. In Section 3, we develop the tools needed for the subsequent sections: we define a notion of saturation with the help of which we introduce the reducibility candidates for the typed calculus. We conclude this section with the correctness theorem, which asserts that our interpretation adheres to typability rules. Furthermore, our definition of saturation implies that the set of typable terms is the only set that is saturated. Therefore, when we intend to prove a certain property for the typable terms, it is enough to prove that the set of typable terms possessing that property is saturated. By our result, it then follows that every typable term enjoys that property. 
In Section 4, we present a proof of the strong normalization of the simply-typed $\l\m$-calculus equipped with $\b\m\r\th\e$-reduction. In Section 5, we take a step further and augment the calculus under consideration with the $\m'$-rule. We demonstrate that the set of terms having a normal form is saturated, which implies that the $\b\m\m'\r\theta\e$-reduction is weakly normalizing. We then conclude with a discussion on future work.

\section{The $\l\mu$-calculus}

The $\l\m$-calculus was introduced by  Parigot \cite{Par1} as a tool for encoding classical natural deduction with terms. In our paper, we restrict our attention to the simply typed calculus, specifically, we are concerned with the representation of the implicational fragment of classical propositional logic using natural deduction style proofs.
 
The $\lm$-calculus involves two sets of variables: one set consists of the original $\l$-variables, often referred to as intuitionistic variables, while the other set comprises $\mu$-variables, known as classical variables.  Parigot's calculus imposes certain restrictions on the rules for forming terms. Notably, a $\mu$-abstraction must immediately be followed by an application involving a $\mu$-variable, and vice versa. However, de Groote \cite{de Gro1} introduced a modified version of the $\lambda\mu$-calculus, which offers more flexibility in its syntax. In de Groote's syntax, a $\mu$-abstraction can be followed by an arbitrary term (in the untyped version), not by a $\mu$-application only. For the purpose of this paper, we exclusively focus on the de Groote-style calculus, and the following definitions pertain only to this version.

\begin{definition}[$\l\m$-terms]~
  \label{def:defterms}
\begin{enumerate}
 \item Let ${\cal V}_{\l} = \{ x,y,z,\dots \}$ denote the set of $\l$-variables and
${\cal V}_{\m} = \{\a ,\b ,\g ,\dots \}$ denote the set of $\mu$-variables,
respectively. The $\l\m$-term formation rules are the following.
\[
{\mathcal{T}} \;\; := \;\; {\cal V}_{\l} \;\; \mid
\;\;  \l {\cal V}_{\l}.{\mathcal{T}}\;\; \mid
\;\;  ({\mathcal{T}}){\mathcal{T}}\;\; \mid
\;\; [{\cal V}_{\m}]{\mathcal{T}} \;\; \mid
\;\; \m {\cal V}_{\m}.{\mathcal{T}} \;\;
\]
We decided to adopt Krivine's notation for the applications, i.e., we write $(M)N$ instead of $(M\, N)$ if we apply $M$ to $N$. 

\item The complexity of a $\l\m$-term is defined inductively as follows:
$cxty(x)=1$, $cxty((M)N)=cxty(M)+cxty(N)+1$ and $cxty(\l x. M) = cxty(\mu \a. M) = cxty([\a]M) = cxty(M)+1$.

\item In a $\l\m$-term the $\l$ and $\mu$ operators bind the variables.
We therefore consider terms modulo equivalence, which allows to rename the variables bound by a $\l$- or a $\mu$-abstraction. The set of free variables in a $\l\m$-term $M$, including both $\l$- and $\m$-variables, will be denoted by $\mathrm{fv}(M)$.


\end{enumerate}
\end{definition}

\begin{definition}[Type system]\label{def:typedsystem}

The types are built from a set ${\cal V}_{\mathfrak{T}}$ of atomic types (propositional variables) and the constant $\bot$ with the connective $\ra$. The type formation rules are the following.
\[
{\mathfrak{T}} \;\; := \;\; {\cal V}_{\mathfrak{T}} \cup \{ \bot \} \;\; \mid \;\;  {\mathfrak{T}} \ra {\mathfrak{T}} \;\;
\]
In the definition below, $\G$ (resp. $\Theta$) denotes a (possibly empty) context, that is, a finite set of declarations of the form $x:A$ (resp. $\a : B$)
for a $\l$-variable $x$ (resp. a $\mu$-variable $\a$) and types $A,B$ such that a $\l$-variable $x$
(resp. a $\m$-variable $\a$) occurs at most once in an expression $x:A$ (resp. $\a :  A$) of $\G$ (resp. of $\Theta$).
The typing rules are as follows.\\

\begin{center}
$\F{}{\G, x:A\;\vdash\; x:A;\Theta}ax$ \\[1cm]
\end{center}

\begin{tabular}{lccl}
$\F{\G, x : A\; \vdash\; M:B;\Theta  }{\G \vdash\; \l x. M:A \ra B;\Theta}\ra_i$   & &  & $\F {\G \;\v\; M:A\ra B;\Theta\;\;\;\;\G \;\v\; N:A;\Theta}{\G \;\v\; (M)N:B;\Theta}\ra_e$\\[1.25cm]

$\F {\G \;\vdash\; M:A,\a : A;\Theta}{\G \;\v\; [\a] M:\bot,\a : A;\Theta }\bot_i$  & & & $\F {\G \;\v\; M:\bot,\a : A;\Theta }{\G \;\v\;\m \a. M:A;\Theta}\bot_e$\\[1cm]
\end{tabular}

We say that the $\l\m$-term $M$ is typable with type $A$, if there is $\G,\Theta$ and a derivation tree such that the uppermost nodes of the tree are axioms and the bottom node is $\G \;\v\; M:A;\Theta$. The presentation with two contexts (one on the left for the $\l$-variables and one on the right for the $\m$-variables) allows us to avoid using negation in types. However, when we correlate typability proofs with proofs in classical natural deduction, the context standing on the right corresponds to a set of hypothesis consisting of negative formulas. Alternatively, we could have applied two contexts on the left: one for the $\l$-variables and one for the $\m$-variables, represented by negated types.

Observe that, in the typed $\l\mu$-calculus, not every term is accepted as well typed. For example, we cannot write a $\l\m$-term of the form $([\a]M)N$ or $\mu \a. \l x.M$.
\end{definition}

In the realm of $\lambda\mu$-terms, various reduction rules exist, with some being fundamental and corresponding to the elimination of logical cuts, while others are introduced to eliminate specific subterms and are often referred to as simplification rules. In this article, we focus on the rules essential for achieving a reasonable representation of data, similar to the case of the $\lambda$-calculus. For a more comprehensive understanding of these rules, readers are encouraged to consult \cite{BN3}, where we demonstrate that Church numerals are the only normal $\lambda\mu$-terms that have the type of integers provided we consider the additional rules $\mu'$, $\rho$, $\varepsilon$ and $\theta$.

Throughout this paper, we employ various types of substitutions. These substitutions include necessary variable renaming to prevent the capture of free variables. Below, we present the concept of $\mu$-substitution, which plays a crucial role in defining $\mu$-redex, as introduced in Definition \ref{def:redex}.

\begin{definition}[$\mu$-substitution]~
\begin{enumerate}
\item  A $\mu$-substitution $\si$ is an expression of the form $[\a:=_sN]$ where $s\in \{l,r\}$, $\a \in {\cal V}_{\mu}$ and  $N \in {\cal T}$.
\item Let $\si$ be the $\mu$-substitution $[\a:=_sN]$ and  $M \in {\cal T}$. We define by induction the $\l\m$-term $M\si$. 
We adopt the convention of renaming bound variables before a substitution so that no variable collision occurs.
Then we can assume that the free variables of the $\l\m$-term $N$ and variable $\a$ are not bound by any $\mu$-abstraction in the $\l\m$-term $M$.
\begin{itemize}
 \item If $M = x$, then $M\si = x$.
 \item If $M = \l x.M'$, then $M\si = \l x.M'\si$.
 \item If $M = (M_1)M_2$, then $M\si = (M_1\si)M_2\si$.
 \item If $M = \mu \b.M'$, then $M\si = \mu \b.M'\si$.
 \item If $M = [\b]M'$ and $\b \neq \a$, then $M\si = [\b]M'\si$.
 \item If $M = [\a]M'$  and $s = r$, then $M\si = [\a](M'\si)N$.
 \item If $M = [\a]M'$  and $s = l$, then $M\si = [\a](N)M'\si$.
\end{itemize}
We adopt the convention that substitution has higher precedence than application and abstraction.
\end{enumerate}
\end{definition}

In order to define $\varepsilon$-reduction, we need the following notion.

\begin{definition}[$\a$-translation]\label{trans}
Let $M \in {\cal T}$ and $\a \in {\cal V}_{\mu}$. We define the $\a$-translation $M_{\a}$ of $M$ by induction on $M$.
\begin{itemize}
\item If $M = x$, then $M_{\a}=x$.
\item If $M = \l x.M'$, then $M_{\a}=\l x.{M'}_{\a}$.
\item If $M = (P)Q$, then $M_{\a}=(P_{\a})Q_{\a}$.
\item If $M = \m \b. M'$, then $M_{\a}=\m \b. {M'}_{\a}$.
\item If $M = [\b]M'$ and $\b \neq \a$, then $M_{\a}=[\b]{M'}_{\a}$.
\item If $M = [\a]M'$, then $M_{\a}={M'}_{\a}$.
\end{itemize}
Intuitively, $M_{\a}$ is the result of replacing every subterm $[\a]N$ in $M$ with $N$.
\end{definition}

We proceed by defining the specific redexes that are the focus of this paper, along with the reductions they induce. Additionally, we provide a brief overview of some of the important results in relation to the $\lambda\mu$-calculus. For a more in-depth understanding of the requisite concepts and definitions, interested readers are referred to the standard textbooks, such as \cite{GLT} and \cite{Kri1}.

\newpage

\begin{definition}[Redex]~\label{def:redex}
\begin{enumerate}
\item A $\beta$-redex is a $\l\m$-term of the form $(\l x. M) N$ and  we call $M[x := N]$ its contractum.
The $\l\m$-term $M[x := N]$ is obtained from $M$ by replacing every free occurrence of $x$ in $M$ by $N$. 
This substitution is sometimes referred to as $\b$-substitution.
\item A $\mu$-redex is a $\l\m$-term of the form $(\m \a. M) N$ and  we call $\m \a.M[\a :=_r N]$ its contractum.
Intuitively, $M[\a :=_r N]$ is obtained from $M$ by replacing every subterm in $M$ of the form $[\a]P$ by $[\a](P)N$.
\item A $\mu'$-redex is a $\l\m$-term of the form $(N)\m \a. M$ and  we call $\m \a.M[\a:=_l N]$ its contractum.
Intuitively, $M[\a :=_l N]$ is obtained from $M$ by replacing every subterm in $M$ of the form $[\a]P$ by $[\a](N)P$.
\item A $\rho$-redex is a $\l\m$-term of the form $[\b]\m \a.M$ and  we call $M[\a := \b]$ its contractum.
The $\l\m$-term $M[\a := \b]$ is obtained from $M$ by replacing every free occurrence of $\a$ by $\b$.
\item A $\theta$-redex is a $\l\m$-term of the form $\m \a.[\a]M$ where $\a \not\in {\rm fv}(M)$ and we call $M$ its contractum.
\item An $\e$-redex is a $\l\m$-term of the form $\m \a.\m \b.M$ and we call $\m \a.M_{\b}$ its contractum.
\end{enumerate}
\end{definition}

The six reductions outlined in Definition \ref{def:redex} find their justification in the cut elimination rules of natural deduction within classical propositional logic. As a consequence of these reductions, the calculus exhibits a type preservation property, as demonstrated in Theorem \ref{SR}, which is also known as the subject reduction property.\\

We will explain how these redex reductions are justified from the typing perspective. The rules $\beta$ and $\mu$ correspond respectively to the intuitionistic and classical cut elimination rules for implication.

\begin{itemize}
\item Rule $\b$: Given the declaration $x:A$, we have $M:B$. If we also have $N:A$, then $(\l x.M)N:B$. It suffices to replace in $M$ all occurrences of $x$ by $N$ to obtain a term of type $B$. This represents the well-known case in which a cut, formed by a subsequent $\rightarrow$-introduction followed by a $\rightarrow$-elimination, is eliminated at the proof level.

\item Rule $\m$: Given the declaration $\a:A \f B$, we have $M : \bot$. If we also have $N:A$, then $(\m \a.M)N:B$. It suffices to start with the declaration $\a : B$ and to replace in $M$ each subterm of the form $[\a]P$ by $[\a](P)N$ to obtain a term of type $B$. We observe that the type associated with $\a$ decreases during this reduction. At the proof level, $\mu$-reduction corresponds to the simplification of a reasoning by contradiction, when a cut formed by a $\bot$-introduction and $\bot$-elimination is replaced by similar cuts of lower type complexity.  

\item Rule $\m'$: This rule does not correspond to a logical cut, but it is the symmetric counterpart of rule $\m$. Given the declaration $\a :A$, we have $M: \bot$. If we also have $N:A \f B$, then $(N)\m \a.M :B$. It suffices to start with the declaration $\a : B$ and to replace in $M$ each subterm of the form $[\a]P$ by $[\a](N)P$ to obtain a term of type $B$. Here, we observe that there is no direct relation between the two types associated with $\a$ before and after the reduction.
\end{itemize}

Rules $\m$ and $\m'$ do not allow elimination of the operator $\m$. However, the three other rules $\r$, $\th$, and $\e$ allow simplification of terms by locally eliminating occurrences of $\m$.

\begin{itemize}
\item Rule $\r$: Given the declaration $\alpha : A$, we have $M : \bot$. If we also have $\beta : A$, then $[\b]\m \a. M : \bot$. It suffices to replace in $M$ all occurrences of $\a$ by $\b$ to obtain a term of type $\bot$.

\item Rule $\th$: Given the declaration $\a : A$, we have $M:A$, hence $\m \a.[\a]M :A$. This term reduces simply to $M$, provided that $\a$ does not appear free in $M$.

\item Rule $\e$: Given the declarations $\a : A$ and $\b : \bot$, we have $M:\bot$, hence $\m \a.\m \b.M :A$. Since both $M:\bot$ and $\m\b.M:\bot$, it suffices to eliminate $\m \b$ as well as all occurrences of $\b$ in $M$ to obtain a term of type $A$. Hence, $\e$-rule removes a proof of $\bot$ by contradiction.
\end{itemize}

\begin{definition} [Reduction and normalization] Let ${\cal R} \subseteq \{\beta, \m, \m', \rho, \theta, \e\}$.
\begin{enumerate}
\item Let $M,M' \in {\cal T}$. We write $M \ra_{{\cal R}} M'$, if $M'$ is obtained from $M$ by replacing an $r$-redex in $M$, where $r \in {\cal R}$,
by its contractum. The reductions (on the redexes) take the following forms (the $\theta$-redex has an additional condition). 

\begin{center}
\begin{tabular}{llll}
   $(\l x.M)N$ & $\ra_{\b}$ & $M[x:=N]$  & \\
   $(\m \a.M)N$  & $\ra_{\m}$ & $\m \a.M[\a:=_rN]$  &\\
   $(N)\m \a.M$  & $\ra_{\m'}$ & $\m \a.M[\a:=_lN]$  &\\
   $[\b]\m \a.M$  & $\ra_{\r}$ & $M[\a:=\b]$  &\\
   $\m \a.[\a]M$  & $\ra_{\th}$ & $M$  &if $\a \not\in {\rm fv}(M)$\\
   $\m \a.\m \b.M$  & $\ra_{\e}$ & $\m \a.M_{\b}$  & \\
\end{tabular}
\end{center}

\item We denote by $\two_{\cal R}$ the reflexive, transitive closure of $\ra_{\cal R}$. I.e., $M \two_{\cal R} M'$ iff $M \ra_{\cal R} M_1\ra_{\cal R}  M_2 \ra \cdots \ra_{\cal R} M_k = M'$.
\item We denote by $\mathcal{NF}_{\cal R}$ the set of all $\l\m$-terms in ${\cal R}$-normal form, i.e., $\l\m$-terms that do not contain an $r$-redex for any $r \in{\cal R}$. 
\item A $\l\m$-term $M$ is said to be ${\cal R}$-weakly normalizable 
if there exists $M'\in \mathcal{NF}_{\cal R}$ such that $M \two_{\cal R} M'$. We denote by $\mathcal{WN}_{\cal R}$ the set of ${\cal R}$-weakly normalizable terms.
\item A $\l\m$-term $M$ is said to be ${\cal R}$-strongly normalizable, if there exists no infinite ${\cal R}$-reduction paths starting from $M$.
That is, any possible sequence of reductions eventually leads to a normal term. We denote by $\mathcal{SN}_{\cal R}$ the set of ${\cal R}$-strongly normalizable terms.
\end{enumerate}
\end{definition}

We list here the most important results that we will use in our paper. 
For the missing proofs, the reader is referred to \cite{BN2,BN3, Dav-Nou3}.

\begin{theorem}[Subject reduction]\label{SR}
The $\b\m\m'\rho\theta\e$-reduction preserves types, that is, if $\G \v M : A; \Theta$ and $M \two_{\b\m\m'\rho\theta\e} M'$, then $\G \v M' : A ; \Theta$.
\end{theorem}

Unfortunately, the previous type preservation property does not carry over to a system based on second-order logic, solely due to the presence of the $\mu'$ rule. To recover such a result, it is necessary to modify the term syntax so as to encode the typing rules for quantifier introduction and elimination. This must also be accompanied by new reduction rules.

\begin{theorem}[Strong normalization for $\b \m\m'$-reduction only]\label{SNmm'}~
\begin{enumerate}
\item The $\m\m'$-reduction is strongly normalizing for the untyped $\l\m$-terms.
\item The $\b \m\m'$-reduction is strongly normalizing for the typed $\l\m$-terms.
\end{enumerate}
\end{theorem}

The preceding results were the first to be established in the presence of the $\mu'$-reduction. The proofs are technical but remain arithmetical and syntactic in nature, relying on a fine-grained analysis of both termination and non-termination phenomena in the calculus. When the additional simplification rules are included, the strong normalization property is lost, and only weak normalization results remain. The proofs we provide also lead to the explicit construction of a normalization algorithm. It is important to point out that, in \cite{Dav-Nou2}, we proved that reducibility candidate techniques fail to establish strong normalization results in the presence of the $\mu'$ rule. Indeed, we provided counterexamples to several lemmas that are often considered essential in this type of proof.

\begin{theorem}[Weak normalization for all reductions]\label{SNmm'rte}~
\begin{enumerate}
\item The $\m\m'\rho\theta\e$-reduction is weakly normalizing for the untyped $\l\m$-terms.
\item The $\b\m\m'\rho\theta\e$-reduction is weakly normalizing for the typed $\l\m$-terms.
\end{enumerate}
\end{theorem}

In this paper, we adopt completely different proof techniques for weak and strong normalizations, based on the method of reducibility candidates. Although the results themselves are already known, the method we propose is new for the $\lambda\mu$-calculus, as it provides a unified approach to obtaining these results.\\

In the sequel, let IH be an abbreviation for the phrase ``the induction hypothesis''.\\

\section{The semantics of the system}

The foundations of realizability methods go back to the work of Kleene \cite{Klee}, who introduced this approach to provide a constructive interpretation of intuitionistic proofs. Tait \cite{Tai} adapted these ideas to elegantly establish the strong normalization of the simply typed $\lambda$-calculus, introducing the well-known candidates of reducibility. This methodology was later deepened and generalized by Girard \cite{Gir} in the context of System $F$, where he significantly extended the applicability of reducibility candidates to prove strong normalization for this more expressive calculus. Krivine \cite{Kri1,Kri5} gave a new impetus to realizability techniques by developing first an intuitionistic, and then a classical, semantics that made it possible to handle proofs by direct computational means. Parigot \cite{Par1} introduced the $\lambda\mu$-calculus, which formally captured classical reasoning within the calculus and opened up a rich field of investigation. His proof of strong normalization \cite{Par4} was based on reducibility candidates: he proposed an elegant method for interpreting types using $\lambda\mu$-terms, relying on the notion of orthogonality.

Several authors have investigated similar techniques to prove strong normalization in various computational systems. Notably, one can cite the works of Barbanera and Berardi \cite{BerBar,BB2} on the $\lambda_{Sym}$-calculus and the $\l_{\mathscr{C}\tau^-}$-calculus, that of Podonovski \cite{Pol} and Herbelin \cite{HH} on the $\lambda\mu\tilde{\mu}$-calculus, and also the contributions of Munch-Maccagnoni \cite{MM} on the $L$-calculus. In each proof of strong normalization, reducibility candidates tailored to the specific syntactic features of the system are introduced, and key properties are established to achieve the intended result. To illustrate this point, consider a type system based on equivalences such as $\neg\neg A = A$: in this context, the use of fixed-point operators becomes essential to obtain equalities between sets of terms, due to the inherent circularity introduced by such type equivalences.

Other authors, such as Pitts \cite{Pit}, Dagand, Rieg and Scherer \cite{DRS}, Herbelin and Miquey \cite{MH}, and Miquel \cite{Miq}, have also contributed to the development of realizability techniques and the use of reducibility candidates, often within theoretical frameworks different from the one we address in this article.

Before introducing our definitions and properties, we wish to draw the readers' attention to the specific difficulties encountered when using reducibility candidates for the simply typed $\lambda\mu$-calculus. The underlying logical system is not as symmetric as those of the systems we have previously mentioned. 
In this calculus, there are no specific objects to designate ``stacks of arguments'', which makes it necessary to formulate definitions using sequences of terms. Moreover, this system contains many simplification rules, which can sometimes pose serious issues, particularly by compromising the strong normalization property in favor of a weaker form of normalization. This situation is unusual in the context of proof techniques based on reducibility candidates, where strong normalization is typically a central goal.

In what follows, we define reducibility candidates for the $\l\m$-calculus and prove its correctness.  First and foremost, we clarify what we mean by a saturated set. We then identify a particular set, denoted as $\B$, which we choose to be saturated. Additionally, we define the properties required of a set to be saturated with respect to $\B$, referred to as being $\B$-saturated. 

Following this, we construct a model consisting of $\B$-saturated sets and interpret the types within this model. In previous works related to similar topics, a specific set of terms was typically fixed from the outset, such as the set of strongly normalizable terms. Subsequently, it was shown that this set satisfied the necessary properties. In contrast, our approach involves identifying the properties of the set $\B$ that are sufficient to establish a correctness theorem. Once these properties are identified, we can turn to concrete examples by showing that certain sets satisfy them and are therefore saturated.

The conditions we impose on the set $\B$ are primarily derived from the typing rules. Our aim is to ensure the fulfillment of the correctness theorem, as indicated by Theorem \ref{adq}. To achieve this, we establish three groups of conditions related to saturation.

\begin{itemize}
\item[-] The first group consists of three straightforward conditions that ensure saturation concerning the constructors $\lambda$, $\mu$, and $[.]$. These conditions contain essential assumptions to maintain the typability of terms. While the first condition does not pose complications in subsequent sections, the other two are more delicate and require special attention.
\item[-] The second group consists of a single condition that is crucial for ensuring that all variables are within the set 
$\B$. This condition is essential for managing non-closed terms (terms containing free variables) and preserving specific properties related to the constructor $\leadsto$, which interprets the type operator, and, ultimately, the logical connective, $\rightarrow$.
\item[-] The final two conditions are relatively standard in the literature and are used in virtually all realizability semantics. They are crucial for managing the typing rules associated with logical introduction rules. It will become apparent later that these conditions, in particular, require that reductions must contain at least the head reduction rules corresponding to $\beta$-reduction and $\mu$-reduction.
\end{itemize}

Due to certain technical difficulties, we made the choice to demand that the set $\B$ exclusively contains typed terms. This choice was found to be justified, as we were going to derive results that only made sense within a typed environment. Consequently, we exercised caution in all our definitions to include only typed terms within the set $\B$.\\

The following definition introduces essential concepts, and Definition \ref{def:RBotgood} outlines the conditions under which a set is deemed saturated.

\begin{definition}[Preliminary notions]~
\begin{enumerate}
\item We denote  the set of typable $\l\m$-terms $\mathcal{T}_t$, i.e.\\
$\mathcal{T}_t = \{M \in \mathcal{T} \; \mid \; \exists \, \G,\Theta,A, \, {\rm such \, that}\;\G \v M:A;\Theta \}.$ 
\item If ${\cal L} \subseteq {\cal T}$, we denote  by $\mathcal{L}^{<\omega}$  the set of finite (possibly empty, denoted by $\varnothing $) sequences of $\l\m$-terms of ${\cal L}$.
\item Let ${\cal L} \subseteq {\cal T}$. If $\bar{N} =N_1 \dots N_n \in \mathcal{L}^{<\omega}$, we write $\bar{P} \tre \bar{N}$ if $\bar{P} = N_1 \dots N_k$ where $0 \leq k \leq n$. 
Intuitively, $\bar{P}$  is an initial subsequence of the sequence $\bar{N}$.
\item Let $M \in  {\cal{T}}$ and $\bar{P} \in \mathcal{T}^{<\omega}$, we define by induction the $\l\m$-term $(M)\bar{P}$ :
$(M)\varnothing = M$ and, if $\bar{P} = N \bar{Q}$, that is, $\bar{P}$ is the sequence consisting of $N\in \mathcal{T}$ followed by $\bar{Q}\in \mathcal{T}^{<\omega}$,  then $(M)\bar{P} =((M)N)\bar{Q}$.
\end{enumerate}
\end{definition}

\begin{definition}[Saturated sets]\label{def:RBotgood}
Let  $\B \subseteq \mathcal{T}_t$. We say that $\B$ is saturated if
\begin{enumerate}
\item[{\rm \bf (C1) :}] $\forall M \in \B $,  $\forall x \in {\cal V}_{\l}$, $\l x.M \in \B$.
\item[{\rm \bf (C2) :}] $\forall M \in \B$, $\forall \a \in {\cal V}_{\m}$,  if $\m \a.M \in {\cal T}_t$,  then $\m \a.M \in \B$.
\item[{\rm \bf (C3) :}] $\forall M \in \B$, $\forall \a \in {\cal V}_{\m}$,  if $[\a]M \in {\cal T}_t$, then $[\a]M \in \B$.
\item[{\rm \bf (C4) :}] $\forall n\geq 0$,  $\forall N_1, \dots, N_n \in \B$, $\forall x \in {\cal V}_{\l}$, if $(x)N_1 \dots N_n \in {\cal T}_t$, then $(x)N_1 \dots N_n \in \B$.
\item[{\rm \bf (C5) :}] $\forall M,N \in {\cal T}_t$, $\forall \bar{P} \in {\cal T}_t^{<\omega}$,  if $N \in \B$, $(\l x.M)N\bar{P}  \in  {\cal T}_t$ and $(M[x:=N])\bar{P} \in \B$, then 
$(\l x.M)N\bar{P}  \in \B$.
\item[{\rm \bf (C6) :}]$\forall M\in {\cal T}_t$, $N \in {\cal T}_t^{<\omega}$, if ${\bar N} \in \B ^{<\omega}$, $(\m \a.M){\bar N} \in   {\cal T}_t$  and $\m \a.M[\a:=_r{\bar N}] \in \B $, then $(\m \a.M){\bar N} \in  \B$.
\end{enumerate}
\end{definition}

\begin{remark}~\rm
\begin{enumerate}
\item The conditions {(C1)}, {(C2)} and {(C3)} describe the requirements for saturation concerning $\l$-, $\m$-abstraction and $\m$-application. We note that we must pay special attention to ensuring that we stay within the realm of typable terms. Consequently, additional conditions are needed to guarantee typability.
\item Condition {(C4)} implies that $\B$ contains $\l$-variables, which will prove to be essential for demonstrating certain closure properties of $\B$. 
\item Condition (C5) states that $\B$ is saturated with respect to head $\b$-reduction when the original term is an application. 
\item Condition (C6) is basically the same but concerns head $\m$-reduction. It is worth noting that a somewhat different condition can be formulated in this case since $\mu$ does not disappear during the reduction.
\item As we proceed, we will see that the stability properties of $\B$ outlined in the above definition are indispensable to the correctness theorem.
\end{enumerate}
\end{remark}

\begin{lemma}
The set ${\cal T}_t$ is saturated.
\end{lemma}

\begin{proof}
We can observe that the conditions of Definition \ref{def:RBotgood} are trivially verified.
\end{proof}

In this section, we are going to prove a surprising result which is the following:
\begin{center}
{\bf The set ${\cal T}_t$ is the unique saturated set. }
\end{center}
We will see in sections 4 and 5 some applications of this powerful result.\\

Now, our goal is to define type interpretations by means of particular subsets of $\B$. To achieve this, it is essential to define our base sets that will serve as the images of type constants, or at the very least, to establish their desired properties. Additionally, we need to introduce an operation, denoted as $\leadsto$, for interpreting the arrow $\to$ in types. However, we must exercise caution in our constructions to ensure that we remain within the domain of typed terms.

The definition of $\leadsto$ provided below is quite standard and aligns with a common concept in realizability semantics: the terms in ${\cal{K}} \leadsto {\cal{L}}$  are those which map the terms of the set ${\cal{K}}$ into the set ${\cal{L}}$. The subtlety resides in formulating the definition so that it is explicitly restricted to typed terms, thereby ensuring that the interpretation remains entirely within the intended domain. Furthermore, our definition exhibits a form of polymorphism: a second interpretation of $\leadsto$ is required to account for the interpretation of the type constructor $\rightarrow$ in cases where $M : A \rightarrow B$ is a $\mu$-abstraction. Crucially, this second interpretation must take into account all initial segments of finite sequences, as this is essential for the proof of Theorem \ref{adq}. 

Following the definition of the $\leadsto$ operator, we outline the conditions imposed on subsets of $\B$ in Definition \ref{def:RBotsat}. These conditions are necessary for interpreting types and ensuring that the underlying subsets of $\B$ remain saturated. 

\begin{definition}[The $\leadsto$ operator]\label{def:leadsto}~
Let $\B$ be a saturated set and consider $\cal{K},\cal{L} \subseteq \B$, and  ${\cal{X}} \subseteq \B^{<\omega}$. 
We define two new subsets of $\B$ relative to $\B$:
$${\cal{K}} \leadsto_{\B} {\cal{L}} =\{ M \in  \B \,\, \mid \,\, \forall \, N \in {\cal{K}},  \, {\rm if} \,  (M) N \in {\cal T}_t, \, {\rm then} \, (M) N \in {\cal{L}}\},$$
$${\cal{X}} \leadsto_{\B} \B =\{M \in  \B \,\, \mid\,\,  \forall \, \bar{N} \in {\cal{X}}, \,  \forall \, \bar{P} \tre  \bar{N}, \, {\rm if} \,  (M)\bar{P} \in {\cal T}_t, \, {\rm then} \, (M)\bar{P} \in \B \}.$$ 
\end{definition}

\begin{definition}[$\B$-saturated sets]\label{def:RBotsat}
Let  $\B$ be saturated and ${\cal S} \subseteq \B$. We call ${\cal S}$ a $\B$-saturated set or, in short, $\B$-saturated if
\begin{enumerate}
\item[{\rm \bf (D1) :}] $\forall M,N \in {\cal T}_t$, $\forall \bar{P} \in {\cal T}_t^{<\omega}$,  if $N \in {\B}$,  $(\l x.M)N\bar{P} \in  {\cal T}_t$ and $(M[x:=N])\bar{P}\in {\cal S} $, then 
$(\l x.M)N\bar{P} \in  {\cal S}$.
\item[{\rm \bf (D2) :}] $\forall n\geq 0$,  $\forall N_1, \dots, N_n \in \B$, $\forall x \in {\cal V}_{\l}$, if  $(x)N_1 \dots N_n \in {\cal T}_t$, then $(x)N_1 \dots N_n \in {\cal S}$.
\item[{\rm \bf (D3) :}] $\exists {\cal X}_{\cal S} \subseteq \B^{<\omega}$, ${\cal S} = {\cal X}_{\cal S}  \leadsto_{\B} \B$. 
\end{enumerate}
\end{definition}

\begin{remark}~\rm
\begin{enumerate}
\item Condition (D1) indicates that $\mathcal{S}$ must be saturated concerning head $\beta$-reduction when the original term is an application. The requirement that $N \in \B$ is essential for this condition. However, we do not need a similar property for saturation with respect to head $\mu$-reduction.
\item Condition (D2) implies that, specifically, every $\l$-variable must belong to $\cal{S}$.
\item Condition (D3) plays a crucial role in allowing us to formulate the correctness theorem (Theorem \ref{adq}). It enables us to transition from an interpretation of a proof of contradiction based on the assumption $\neg A$ to an interpretation of a proof of $A$.
\item We note that, if $\B$ is a saturated set, it is also considered $\B$-saturated since $\B = \varnothing \leadsto_{\B} \B$.
\end{enumerate}
\end{remark}

The following remark can be seen as an easy exercise and will not be used in the rest of the paper.

\begin{remark}~\rm
Let  $\cal{K},\cal{L}$ be  $\B$-saturated, $M \in {\cal L}$ and $x \in {\cal V}_{\lambda}$  such that $x$ is not free in $M$, then, by {\rm (C1)}, $\l x.M \in \B$. We have $\forall N\in {\cal K}$, if $(\l x. M)N \in {\cal T}_t$, then $M[x:=N] = M \in {\cal L}$, thus, by {\rm (D1)},  $(\l x. M)N \in  {\cal L}$. Therefore  $\l x.M  \in {\cal{K}} \leadsto_{\B} {\cal{L}}$. 
This shows that we can create an infinite number of terms in ${\cal{K}} \leadsto_{\B} {\cal{L}}$. We will see in Lemma \ref{ModelSat} that 
 ${\cal{K}} \leadsto_{\B} {\cal{L}}$ contains  every $\l$-variable.
\end{remark}

In the sequel, we omit the subscript ${\B}$ and we simply write $\cal{K}\leadsto \cal{L}$ or $\cal{X}\leadsto \B$.\\

We can now introduce the concept of a $\B$-model, which comprises the sets used for interpreting types. Lemma \ref{ModelSat} is of paramount importance as it guarantees that the sets comprising a model maintain their property of being $\B$-saturated. As we shall see, we allow a $\B$-model to be additionally parameterized by a family of $\B$-saturated sets. Although we will not explore the relevance of this generalization in the present paper, we mention it to emphasize that, in a setting where one aims to establish properties other than normalization, such sets could also be used to interpret types and derive further properties of the calculus.

\begin{definition}[$\B$-model]
Let $\B$ be a saturated set and ${({\cal S}_i)}_{i \in I}$ a family of $\B$-saturated sets. 
A $\B$-model ${\cal M}$ is the smallest set defined by the following inductive steps.
\begin{enumerate}
\item $\B\in {\cal M}$ and, for all $i \in I$, ${\cal S}_i\in {\cal M}$.
\item If ${\cal U}$, ${\cal V}\in {\cal M}$, then ${\cal U}\leadsto{\cal V}\in {\cal M}$.
\end{enumerate}
That is, ${\cal M}$ denotes the smallest set that contains $\B$ and all ${\cal S}_i$, and is closed under the constructor $\leadsto$.
 We shall not explicitly mention the family ${({\cal S}_i)}_{i \in I}$ of $\B$-saturated sets as a parameter of the model, in order to avoid overloading the notation.
\end{definition}

It will turn out (notably in  Corollary \ref{closed}) that, for the correctness theorem to hold, it would be enough to consider models containing only the set $\B$. At this point, however, we do not know yet that ${\cal T}_t$ is the unique saturated set. 

\begin{lemma} \label{ModelSat}
Let ${\cal M}$ be a $\B$-model. Then the following are valid
\begin{enumerate}
    \item Let ${\cal U}, {\cal V}\in \mathcal{M}$ be $\B$-saturated. Then ${\cal U}\leadsto{\cal V}\in \mathcal{M}$ implies that ${\cal U}\leadsto{\cal V}$ is $\B$-saturated.
    \item Every element of $\mathcal{M}$ is $\B$-saturated.
\end{enumerate}
\end{lemma}

\begin{proof}
Let ${\cal M}$ be a $\B$-model. We verify both points of the lemma in the order presented above.
\begin{enumerate}
\item Let ${\cal U},{\cal V} \in {\cal M}$ be $\B$-saturated such that ${\cal U}\leadsto{\cal V}\in \mathcal{M}$. We verify the three points of Definition \ref{def:RBotsat} for ${\cal U} \leadsto {\cal V}$.
\begin{itemize}
    \item[{\bf (D1)}] Let $M,N \in {\cal T}_t$ and $\bar{P} \in {\cal T}_t^{<\omega}$ such that $N \in \B$, $(\l x.M)N\bar{P} \in  {\cal T}_t$ and $(M[x:=N])\bar{P}\in {\cal U}\leadsto{\cal V}$.
    We will prove that $(\l x.M)N\bar{P} \in  {\cal U} \leadsto {\cal V}$.
    First, we have  $(M[x:=N])\bar{P}\in \B$ (by virtue of Definition \ref{def:leadsto}), then, by (C5), $(\l x.M)N\bar{P} \in  \B$.
    Let $Q \in {\cal{U}}$ such that $(\l x.M)N\bar{P}Q\in \mathcal{T}_t$, then $(M[x:=N])\bar{P}Q\in \mathcal{T}_t $. Since we have $(M[x:=N])\bar{P}\in \B$ and $(M[x:=N])\bar{P} \in {\cal U} \leadsto {\cal V}$, thus $(M[x:=N])\bar{P}Q\in \cal{V}$ and, by assumption, $(\l x.M)N\bar{P}Q \in \cal{V}$. Therefore $(\l x.M)N\bar{P} \in {\cal U} \leadsto {\cal V}$.  

    \item[{\bf (D2)}] Let $ N_1, \dots, N_n \in \B$ and $x \in {\cal V}_{\l}$ such that  $(x)N_1 \dots N_n \in {\cal T}_t$.  
    We will prove that $(x)N_1 \dots N_n  \in  {\cal U} \leadsto {\cal V}$. First, by  (C4), $(x)N_1 \dots N_n  \in  \B$.
    Assume $N \in \cal{U}$ and $(x)N_1 \dots N_nN \in {\cal T}_t$, then, since $\mathcal{U}$ and $\mathcal{V}$ are $\B$-saturated, $N \in \B$ and $(x)N_1 \dots N_nN \in \cal{V}$. 
    Therefore $(x)N_1 \dots N_n \in {\cal U} \leadsto {\cal V}$.  

    \item[{\bf (D3)}] Let  ${\cal X}_{\cal V} \subseteq \B^{<\omega}$ such that ${\cal V} = {\cal X}_{\cal V}  \leadsto \B$. By assumption, such an ${\cal X}_{\cal V}$ exists.
    Let  ${\cal U}{\cal X}_{\cal V}  = \{ N\bar{N} \, \mid \, N \in \cal{U}$ and $\bar{N} \in {\cal X}_{\cal V} \}$. We will verify ${\cal U} \leadsto {\cal V} =  \cal{U}{\cal X}_{\cal V} \leadsto \B$.
    First, we have $ \cal{U}{\cal X}_{\cal V} \subseteq \B^{<\omega}$.
    \begin{itemize}
        \item Let $M \in {\cal U} \leadsto {\cal V}$ and $\bar{P} = N \bar{Q}$ where $N \in \cal{U}$ and  $\bar{Q} \in {\cal X}_{\cal V}$. Let  $\bar{R} \tre \bar{P}$  such that $(M)\bar{R} \in {\cal T}_t$. If $\bar{R} = \varnothing$, then $(M)\bar{R} = M \in \B$.  If not,  $\bar{R} =  N \bar{S}$ where  $\bar{S} \tre \bar{Q}$.
        Since $(M)N \in {\cal T}_t$, $(M)N \in {\cal V} \subseteq \B$ and since $(M)N\bar{S} \in {\cal T}_t$ and $\bar{S}\tre\bar{Q}\in {\cal X}_{\cal V}$,  $(M)N\bar{S} \in \B$. 
        Therefore ${\cal U} \leadsto {\cal V} \subseteq  \cal{U}{\cal X}_{\cal V} \leadsto \B$.

        \item Let $M \in  \cal{U}{\cal X}_{\cal V} \leadsto \B$ and $N \in \cal{U}$ such that $(M)N \in {\cal T}_t$, then $(M)N \in \B$, since $N\tre N\bar{Q} \in \cal{U}{\cal X}_{\cal V}$. 
        Let  $\bar{P} \tre \bar{N} \in {\cal X}_{\cal V}$ such that $(M)N\bar{P} \in {\cal T}_t$, then $N\bar{P} \tre N \bar{N}$, thus $(M)N\bar{P} \in \B$. Hence, $(M)N \in {\cal X}_{\cal V} \leadsto \B=\cal{V}$. 
        Therefore $\cal{U}{\cal X}_{\cal V} \leadsto \B \subseteq {\cal U} \leadsto {\cal V}$.
    \end{itemize}
\end{itemize}
\item Let \({\cal S} \in {\cal M}\). We proceed by induction on the number of \(\leadsto\)-steps in the construction of \({\cal S}\). If \({\cal S} = \B\) or \({\cal S} = {\cal S}_i\), the statement holds trivially. Now, assume \({\cal S} = {\cal U} \leadsto {\cal V}\) for some \({\cal U}, {\cal V} \in {\cal M}\). By IH, both \({\cal U}\) and \({\cal V}\) are \(\B\)-saturated. The conclusion then follows from the first point of the lemma.
\end{enumerate}
\end{proof}

We have now reached a pivotal concept in grasping the essence of these kind of reducibility candidates. When attempting to model the equation $A = \neg \neg A$, rather than relying on fixed-point operators, we embrace the realizability semantics proposed by Parigot. Specifically, he addressed the challenge of deriving the interpretation of $A$ from that of $\neg \neg A$ by ensuring that every element $\mathcal{S}$ within $\mathcal{M}$ can be represented as $\mathcal{S}^\perp \leadsto \B$ for some set $\mathcal{S}^\perp$. Intuitively, we can consider $\mathcal{S}^\perp$, referred to as the orthogonal of $\mathcal{S}$, as the interpretation of the negation of a type. This concept enables us to effectively incorporate classical logic without the need for fixed-point operators.\\

In the subsequent definition, we provide the weakest precondition $\mathcal{W}$ for a set $\mathcal{S}\in  \mathcal{M}$ such that $\mathcal{S}=\mathcal{W}\leadsto \B$ holds. Following Parigot \cite{Par4}, we adopt the notation $\mathcal{S}^{\B}$ for $\mathcal{W}$. The fact that the $\mu$-operator does not vanish during a reduction is reflected in the second item of Definition \ref{def:leadsto}. In Definition \ref{def:leadsto}, however, we focus on the functional part of a $\mu$-redex rather than its arguments, since, in the proof of the adequacy lemma, we must reconstruct the candidate for the functional part from those of the arguments. This reconstruction will be made possible by Lemma \ref{orth2}. Our definition of  $\mathcal{S}^{\B}$ provides somehow ``the reverse'' notion compared to Definition \ref{def:leadsto}: we focus on the finite sequences, ``the orthogonals'', that enable us to verify the adequacy property also for the $\m$-abstraction. For the properties we aim to establish, it is vital that the empty list be included in the orthogonal set of the interpretation of each type. Therefore, we incorporate it into $\mathcal{S}^{\B}$ in the subsequent definition.

\begin{definition}[Weakest precondition]\label{def:weakestprecondition} 
Let ${\cal M}$ be a $\B$-model, and ${\cal S} \in  {\cal M} $. We write 
$${\cal{S}}^{\B} =\left(\bigcup\{{\cal{X}} \subseteq \B^{<\omega} \,\, \mid \,\, \cal{S}= \cal{X} \leadsto \B \}\right) \cup \{ \varnothing \}.$$
\end{definition}

The next lemma demonstrates that $\cal{S}^{\B} \leadsto\B$ coincides with $\cal{S}$, hence, $\cal{S}^{\B}$ does indeed describe the weakest condition under which can ensure this property. 

\begin{lemma} \label{orth2}
Let ${\cal M}$ be a $\B$-model, and ${\cal S} \in  {\cal M}$. We have ${\cal S} = {\cal S}^{\B} \leadsto \B$.
\end{lemma}

\begin{proof} 
\begin{itemize}
\item Assume $M \in {\cal S} $ and $\bar{P} \tre \bar{N} \in \cal{S}^\perp $ such that $(M)\bar{P} \in {\cal T}_t$.
\begin{itemize}
\item If $\bar{N} = \varnothing$, then $\bar{P} = \varnothing$ and $(M)\bar{P}  = M \in {\cal S}  \subseteq \B$. 
\item If  $\bar{N}\in {\cal X}_0$ for some ${\cal X}_0$ with  ${\cal S}= {\cal X}_0 \leadsto \B$, we have, by definition, $(M)\bar{P}\in \B$. 
\end{itemize}
\item Let ${\cal X}_0 \subseteq \B^{<\omega}$ such that  ${\cal S} = {\cal X}_0  \leadsto \B$. 
Since $ {\cal X}_0  \subseteq {\cal S}^{\B}$, then  ${\cal S}^{\B} \leadsto \B \subseteq {\cal X}_0  \leadsto \B = {\cal S}$.
\end{itemize}
\end{proof}

We can now interpret the types in a $\B$-model. Note that the interpretation of the type $\bot$ will be the whole set $\B$.

\begin{definition}[Interpretation of types] 
Let ${\cal M}$ be a $\B$-model.
An $\mathcal{M}$-interp-retation $\cal{I}$ is a function $X \mapsto {\cal{I}}(X)$ from the set of atomic types $\mathcal{V}_{\mathfrak{T}}$ to $\mathcal{M}$ which we extend for any type formula as follows: ${\cal{I}}(\bot)=\B$ and
${\cal{I}}(A \to B)= {\cal{I}}(A) \leadsto {\cal{I}}(B)$.
\end{definition}

We are now prepared to state and prove the generalized correctness theorem. This theorem establishes a connection between the concept of typability and interpretation within a model. In simple terms, it asserts that any $\lambda\mu$-term $M$ of type $A$ belongs to the interpretation of $A$ in any model. Furthermore, given that the interpretation of any type is contained in $\B$, it follows that the $\lambda\mu$-term $M$ will also be in $\B$. Consequently, a well-chosen set $\B$ enables us to uncover properties of typable terms in relation to a specific reduction.

\begin{definition}[Simultaneous substitution]
Let $M,M_1,\dots ,M_n \in \cal{T}$, $\bar{N}_1,\dots$ $,\bar{N}_m \in \mathcal{T}^{<\omega}$, and $\s $ be the simultaneous substitution 
$[(x_i:=M_i)_{1 \le i\le n};(\a_j:=_r\bar{N}_j)_{1 \le j \le m}]$, which is not an object of the syntax. 
Then $M\s$ is obtained from the $\l\m$-term $M$ by replacing each $x_i$ by $M_i$ and replacing inductively each subterm of the form
$[\a_j]U$ in $M$ by $[\a_j](U)\bar{N}_j$. We explicitly state only the cases in which the specific behavior of a \(\mu\)-substitution arises. Namely, when \( M \) is of the form \([\alpha]U\) for some \(\alpha\) and \(U\). The remaining cases proceed in the usual inductive manner and here are the two important cases:
\begin{itemize}
\item If $M = [\alpha]U$ and $\alpha \neq \alpha_j$, then $M\s = [\alpha]U\s$.
\item If $M= [\alpha_j]U$, then $M\s = [\alpha_j](U\s)\bar{N}_j$.
\end{itemize}
\end{definition}



The following result establishes a connection between typing and semantics. Intuitively, it states that a typable term always belongs to its interpretation in a model. Of course, one must properly handle the free variables of the term by substituting them with appropriate terms. It is worth noting that the statement of the theorem relies on a set of substantial assumptions. In particular, initial subsequences of sequences are included in ${\cal I}(B_j)^{\B}$, and the term $M\s$ is assumed to be typable.
A careful examination of the proof shows that these assumptions are indispensable for the current approach.

\begin{theorem} [General correctness theorem] \label{adq} $\;$\\
Let  ${\cal M}$  be a $\B$-model,
${\cal{I}}$ an $\mathcal{M}$-interpretation, \\
$\G =\{x_i : A_i\}_{1\leq i\leq n}$, 
$\Theta =\{\a_j: B_j\}_{1\leq j\leq m}$,\\
$M_i \in \;{\cal{I}}(A_i)$ for all  $1\leq i\leq n$, 
$\bar{P_j} \tre \bar{N_j} \in\; ({\cal{I}}(B_j))^{\B}$ for all $1\leq j\leq m$ and \\
$\sigma =[(x_i:=M_i)_{1 \leq i \leq n};(\a_j:=_r\bar{P_j})_{1\leq j\leq m}]$.\\
If $\Gamma \vdash M:A;\Theta$ and $M\sigma \in {\cal T}_t$ , then $M\sigma \in {\cal{I}}(A)$.
\end{theorem}

\begin{proof} 
By induction on the  derivation, we consider the last rule used.
\begin{itemize}

\item[$ax:$] In this case, $M = x_i$, $A = A_i$, $\Gamma \vdash x_i:A_i;\Theta$ and  $M \s = \si(x_i) = M_i \in  {\cal T}_t$ . By hypothesis, $M= M_i \in {\cal I}(A_i) =  {\cal I}(A)$.

\item[$\ra_i:$] In this case, $M = \l x.N$ and $A = B\ra C$, $\Gamma, x : B \vdash N : C;\Theta$ and  $M \s = \l x.N\s \in {\cal T}_t$.
Since $x\in {\cal I}(B)$ and $N\s = N\s'$ where $\s' = \s + [x:=x]$, we obtain, by IH, $N\s  \in {\cal I}(C)\subseteq \B$, then (C1)  involves $M\s \in \B$. 
Let $P \in {\cal I}(B)$ such that $(M\s)P \in {\cal T}_t$.  We have $(M\s)P  \ra_{\b} N\s''$ where $\s'' = \s + [x:= P]$, then $N \s'' \in {\cal T}_t$.
By IH, $N\s'' \in {\cal I}(C)$ and $P \in \B$, then, by (D1), $(M\si)P\in {\cal I}(C)$. Therefore $M\s \in {\cal I}(B)\leadsto {\cal I}(C)={\cal I}(A)$.  

\item[$\ra_e:$] In this case, $M = (N)P$, $\G \v P: A;\Theta$, $\G \v N: B\ra A;\Theta$, and $M\s = (N\s)P\s \in {\cal T}_t$. Then 
$N\s,P\s \in {\cal T}_t$ and, by IH,  $N\si \in {\cal I}(B)\leadsto {\cal I}(A)$ and $P\si \in{\cal I}(B)$. This means $M\si = (N\si)P\si\in {\cal I}(A)$.

\item[$\bot_i:$]  In this case, $M = [\a_j]N$, $A=\bot$, $\Gamma\v N:B_j;\Theta $ and let $\bar{P_j} \tre \bar{N_j} \in {\cal I}(B_j)^{\B}$ be such that $M\s = [\a_j](N\s)\bar{P_j} \in {\cal T}_t$. Then $N\s ,  (N\s)\bar{P_j} \in {\cal T}_t$ and,  by IH, $N\si\in {\cal I}(B_j)$  and, since  $\bar{P_j} \tre \bar{N_j} \in {\cal I}(B_j)^{\B}$, then  $(N\si)\bar{P_j}\in \B$. Therefore, by (C3), $M\si=[\a_j](N\si)\bar{P_j}\in \B ={\cal I}(\bot) $.

\item[$\bot_e:$] In this case, $M=\mu \a.N$, $\Gamma\vdash N:\bot;\, \a:A;\Theta'$ and $M\s = \m \a.N\s \in {\cal T}_t$. Since $\varnothing  \in {\cal I}(A)^{\B}$ and $N\s = N\s'$ where 
$\s' = \si + [\a :=_r \varnothing]$, we obtain, by IH, $N\s \in {\cal I}(\bot) = \B$.
Let $\varnothing  \neq \bar{P} \tre {\bar N}\in {\cal I}(A)^{\B}$ such that $(M\s){\bar P} \in {\cal T}_t$.  We have $(M\s){\bar P} \two_{\m} \mu \a. N\s''$ where  $\s'' = \si + [\a :=_r {\bar P}]$, then $\m\a.N\s''$, $N\s'' \in {\cal T}_t$.  By  IH, $N\s'' \in \B=\mathcal{I}(\bot)$ and, by (C2), $\m\a. N\s'' = \m\a.N\s[\a:=_r {\bar P}]\in \B$, hence, by (C6), $(M\si){\bar P}  \in \B$. Therefore $M\si=\m\a.N\si\in {\cal I}(A)$.
\end{itemize}
\end{proof}

We can now state and prove the main result of this section.

\begin{corollary}\label{closed}
 For every saturated set $\B$ we have $\B = {\cal T}_t$, i.e., ${\cal T}_t$ is the unique saturated set.
\end{corollary}

\begin{proof} 
It suffices to check that if $\G \vdash M:A;\Theta$ and $\B$ is  a saturated set, then $M \in \B$.
Assume $\G =\{x_i : A_i\}_{1\leq i\leq n}$ and $\Theta =\{\a_j: B_j\}_{1\leq j\leq m}$.
Let  ${\cal M} $ be the $\B$-model containing only the set $\B$ and 
${\cal{I}}$ the $\mathcal{M}$-interpretation defined by $\cal{I}$$(X) = \B$ for all $X \in \mathcal{V}_{\mathfrak{T}}$. Since $x_i \in {\cal{I}}(A_i)$ for all $1 \leq i \leq n$ and
$\varnothing  \in {\cal{I}}(B_j))^{\B}$ for all $1 \leq j \leq m$, then,  by the general correctness lemma and since $M \in {\cal T}_t$, $M = M[(x_i:=x_i)_{1 \leq i \leq n};(\a_j:=_r \emptyset)_{1\leq r\leq j}] \in {\cal I}(A)\subseteq \B$.
\end{proof}

$\;$



The preceding result shows that, interestingly, the set of typable $\lambda\mu$-terms is, in fact, the only set of terms satisfying the conditions that define a saturated set in the sense of Definition \ref{def:RBotgood}. In other words, to prove that a certain set of typable $\lambda\mu$-terms satisfies a given property, it suffices to consider the set of terms satisfying that property and to show that this set is saturated. In Sections 4 and 5, we explore applications of this result.

\section{Strong normalization property of $\b\m\r\e\theta$-reduction}

In this section, we demonstrate that $\b\m$-reduction enjoys the strong normalization property when augmented with the rules $\r,\e$, and $\theta$, i.e., all the mentioned rules except for $\mu'$. 
This result is not new: Parigot originally established it for his calculus, and de Groote later confirmed it for his own variant. What distinguishes our approach is the methodology employed. Rather than proving strong normalization for a restricted subset of rules and subsequently extending the result via commutation arguments, we directly verify the property by considering the full set of reduction rules. Specifically, we take as our base set the set of typable terms that are strongly normalizable with respect to all the underlying rules. We then demonstrate that this set satisfies all the necessary conditions to be considered saturated.

Accordingly, we define ${\cal R} = \{\beta, \mu, \rho, \theta, \eta\}$ and set $\B = \mathcal{SN}_{\cal R} \cap \mathcal{T}_t$. Our goal is to verify that $\B$ is saturated. By Corollary \ref{closed}, this will establish that every typable term is strongly normalizable with respect to $\beta\mu\rho\theta\eta$-reduction. In the following discussion, we adopt the abbreviation $M[\alpha := N]$ for $M[\alpha :=_r N]$, as $\mu'$-reduction is not considered in this context. We then proceed to verify that the conditions specified in Definition 
\ref{def:RBotgood} are satisfied.\\

First of all, we state a few lemmas.

\begin{lemma}\label{triva}
Let $M$, $N\in {\cal T}$ and $\a, \b, \g\in {\cal V}_\m$ such that $\a\neq \b, \g$. 
Then $(M[\b:=\g])_{\a} = M_{\a}[\b:=\g]$, $(M[x:=N])_{\a} = M_{\a}[x:=N_\a]$,  $(M[\b:=N])_{\a} = M_{\a}[\b:=N_\a]$, and  $(M_\a)_\b = (M_\b)_\a$.
\end{lemma}

\begin{proof}
All the cases follow by a straightforward induction on $M$.
\end{proof}

\begin{lemma} Let $M\in {\cal T}$ and $\a,\b\in {\cal V}_{\m}$. Then $M\ra_{\cal R} N$ iff $M[\a:=\b]\ra_{\cal R} N[\a:=\b]$. 
\end{lemma} 

\begin{proof}
By induction on $M$.
\end{proof}

\begin{lemma}\label{alphabetasn}
Let  $M\in {\cal T}$ and $\a,\b\in {\cal V}_{\m}$. Then $M\in {\cal NF}_{\cal R}$ iff $M[\a:=\b]\in {\cal NF}_{\cal R}$.
\end{lemma}

\begin{proof}
Follows from the previous lemma.
\end{proof}

Now we turn to verifying the properties listed in Definition \ref{def:RBotgood}. 
 
\begin{lemma}[Condition (C1)]\label{C1sn}
Let $M\in \B$ and $x \in {\cal V}_\l$. Then $\l x.M\in \B$.
\end{lemma}

\begin{proof}
Straightforward.
\end{proof}

Before proving the condition {(C2)}, we point out that it is inevitable that we work in the typed framework. Indeed, it is easy to find an example of a $\l\m$-term $M \in {\cal SN}_{\cal R}$ but $\m\a.M \not\in {\cal SN} _{\cal R}$.
Namely, let $M = \m\b.([\b]\l x. (x)x)\l x. (x)x$, then we have $M \in {\cal NF}_{\cal R}$ but $\m\a.M \ra_{\e} \m\a.(\l x. (x)x)\l x. (x)x \not\in\mathcal{WN}_{\cal{R}}$. More generally, the following remark will be valid throughout the section.

If we choose $\B$ as above, we can observe that none of the conditions of Lemma \ref{def:RBotgood} will necessitate a typed framework, except for the condition (C2). However, in adherence to the formulation of Lemma \ref{def:RBotgood}, we include typability assumptions in the respective lemmas. Only the verification of condition (C2) for $\B$ will depend on typability assumptions.

\begin{lemma}[Condition (C2)]\label{C2sn}
Let $M\in \B$ and $\a \in {\cal V}_\m$ such that  $\m\a.M\in {\cal T}_t$. Then $\m\a.M\in \B$.
\end{lemma}

\begin{proof}
We prove the lemma with the help of Lemmas \ref{Malphasn} and \ref{etaMalpha}.
\end{proof}

\begin{lemma}\label{Malphasn}
Let $M\in {\cal T}_t$ and assume $M_\a\ra_{\cal R} U$ for some $\a \in {\cal V}_\m$. Then $\exists$ $V$ such that $M\ra_{\cal R} V$ and $V_\a=U$. 
\end{lemma}
 
\begin{proof}
Let $M\in {\cal T}_t$ and assume $M_\a\ra_{\cal R} U$.
\begin{itemize}
\item If $M$ is a variable, there is nothing to prove.
\item If $M=\l x.N$, we can apply IH. 
\item Assume $M=[\b]N$.  
\begin{itemize}
\item If $\b\neq \a$, then $M_{\a}=[\b]N_{\a}$. 
If $U=[\b]W$ with $N_{\a}\ra_{\cal R} W$, then we can apply IH. 
Otherwise, $N=\m\g.P$, and, applying Lemma \ref{triva}, $[\b]N_{\a}=[\b]\m\g.P_{\a}\ra_\r P_{\a}[\g:=\b]=(P[\g:=\b])_{\a}$. We let $M \ra_\r P[\g:=\b]=V$. 
\item If $\b=\a$, then $M_{\a}=([\a]N)_{\a}=N_{\a}$. In this case we use IH.
\end{itemize}
\item Assume $M=\m\g.N$. 
\begin{itemize}
\item If $U=\m\g.W$ such that $N_{\a}\ra_{\cal R} W$, then we can apply IH. 
\item If $N=\m\delta.P$, then, by Lemma \ref{triva}, $M_{\a}=\m\g.\m\delta.P_{\a}\ra_\varepsilon \m\g.(P_{\a})_{\delta}=\m\g.(P_\delta)_{\a}$ and we let $M \ra_\varepsilon \mu\g.P_{\delta}=V$. 
\item If $N=[\g]P$, $\g \notin {\rm fv}(P)$, $M_a =\m \g.[\g]P_\a \ra_\th P_\a$ and we have $M \ra_\th P=V$.
\end{itemize}
\item If $M=(N)P$, then $M_{\a}=(N_{\a})P_{\a}$. If $U=(Q)R$ such that $N_{\a}\ra_{\cal R} Q$ or $P_{\a}\ra_{\cal R} R$, then we can apply IH. 
Hence, we may assume that $M_\a$ is the redex reduced. We distinguish the various cases according to the structure of $N$. 
\begin{itemize}
\item $N$ cannot be a variable, since then $M_\a$ is not a redex.
\item If $N=[\g]N'$, then $N$ has type $\bot$, hence this case is impossible.
\item If $N=\l x.N'$, then $M \ra_\b N'[x:=P]$, and, by applying Lemma \ref{triva}, $M_\a=(\l x.N'_{\a})P_{\a}\ra_\b N'_{\a}[x:=P_{\a}]=(N'[x:=P])_{\a}$, and the result follows.
\item If $N=\m\g.N'$, then $M \ra_\m\m\g.N'[\g:=P]$, and, by Lemma \ref{triva}, $M_\a=(\m\g.N'_{\a})P_{\a}\ra_\m \m\g.N'_{\a}[\g:=P_{\a}]=(\m\g.N'[\g:=P])_{\a}$, and the result follows.  
\item The case $N=(N_1)N_2$ is impossible, since then $M_\a$ is not a redex.
\end{itemize}
\end{itemize}
\end{proof}

\begin{definition}[Strong normalization]
We say that a term $M$ is strongly normalizable, in notation $M \in {\cal SN}_{\cal R}$, if every reduction sequence starting from $M$ terminates. Since the reduction tree is locally finite, then, by K{\"o}nig's lemma, this is equivalent to asserting 
that the lengths of the reduction sequences starting from $M$ is bounded from above. Let us denote the length of the longest reduction sequence of $M$ by $\eta(M)$. Observe that, if $M \ra_{\cal R} M'$, then $\eta(M) > \eta(M')$. 
Then  $M \in {\cal SN}_{\cal R}$ iff, for every $M \ra_{\cal R} N$, we have $N \in {\cal SN}_{\cal R}$. Moreover, $\eta(M)=max\{\eta(N)\mid M\ra_{\cal R} N\}+1$.
\end{definition}

\begin{lemma}\label{etaMalpha}
Let $M\in {\cal T}_t$ and $\a\in {\cal V}_\m$. If  $M \in {\cal SN}_{\cal R}$, then $M_\a\in {\cal SN}_{\cal R}$ and $\eta(M_\a) \leq \eta(M)$.
\end{lemma}

\begin{proof}
Let $M\in {\cal SN}_{\cal R}$ and assume $M_\a\ra_{\cal R} U$. We reason by induction on $\eta(M)$. By the previous lemma, $\exists$ $V$ such that $M\ra_{\cal R} V$ and $V_\a =U$.  
Then $\eta(V)< \eta(M)$, hence, we can apply IH to $V$. We obtain the bound for $\eta(M_\a)$ if we take into account the inequality $\eta(U)+1\leq \eta(V)+1\leq \eta(M)$, which follows from the previous fact and IH. 
\end{proof}

We can now turn to the proof of Lemma \ref{C2sn}.\\

\noindent\textbf{Proof of Lemma \ref{C2sn}.}\\
Let $M\in \B$ and $\a \in {\cal V}_\m$ such that  $\m\a.M\in {\cal T}_t$. We assume by lexicographic induction with respect to  
$(\eta(M), cxty(M))$ that $M \in {\cal SN}_{\cal R}$ implies $\m\a.M \in {\cal SN}_{\cal R}$. Let $\m \a.M \ra_{\cal R} N$. We claim that $N\in {\cal SN}_{\cal R}$.
We examine the three cases possible.
\begin{itemize}
\item If $N=\m\a.M'$ with $M \ra_{\cal R} M'$, then $\eta(M')<\eta(M)$ and, by IH, $N \in {\cal SN}_{\cal R}$.
\item If $M = \m \b. P$ and $\m\a.M \ra_{\varepsilon} \m \a. P_\b = N$, then $P \in {\cal SN}_{\cal R}$, $\eta(P_\b) \leq \eta(P) \leq \eta(M)$ and $cxty(P) < cxty(M)$, thus, by IH, $N \in {\cal SN}_{\cal R}$.
\item  If $M = [\a] N$, $\a \notin {\rm fv}(N)$ and $\m\a.M \ra_{\th} N$, then $N \in {\cal SN}_{\cal R}$.
\end{itemize} 
\hfill$\square$

\begin{lemma}[Condition (C3)]\label{C3sn}
Let $M\in {\B}$  and $\a \in {\cal V}_\m$ such that $[\a]M\in {\cal T}_t$. Then $[\a]M\in \B$.
\end{lemma}

\begin{proof}
Let $M\in \B$ and $\a \in {\cal V}_\m$ such that  $[\a]M\in {\cal T}_t$. We will prove, by induction on $\eta(M)$, that $[\a]M \in {\cal SN}_{\cal R}$.
Let $[\a]M \ra_{\cal R} N$. It suffices to check that $N \in {\cal SN}_{\cal R}$.
We examine two cases.
\begin{itemize}
\item If $N=[\a]M'$ with $M \ra_{\cal R} M'$, then $\eta(M')<\eta(M)$ and, by IH, $N \in {\cal SN}_{\cal R}$.
\item If $M = \m \b. P$ and $[\a]M \ra_\r P[\b:=\a] = N$, then $P \in {\cal SN}_{\cal R}$, thus, by Lemma \ref{alphabetasn},
$N \in {\cal SN}_{\cal R}$.
\end{itemize} 
\end{proof}

Condition {(C4)} proves to be trivial due to the fact that we have omitted the $\m'$-reduction, that is, the symmetric counterpart of $\m$.

\begin{lemma}[Condition (C4)] \label{C4sn}
Let $n\geq 0$, $ N_1, \dots, N_n \in \B$ and $x \in {\cal V}_{\l}$. Then $(x)N_1 \dots N_n \in {\cal T}_t$ involves $(x)N_1 \dots N_n \in \B$.
\end{lemma}

\begin{proof}
We argue by induction on $\displaystyle\sum_{k=1}^{n}\eta(N_k)$. Indeed, $(x)N_1 \dots N_n\ra_{\cal R} U$ implies $U=(x)U_1 \dots U_n$, where $N_i\ra_{\cal R} U_i$ for one index $1\leq i\leq n$ and $N_j=U_j$ if $1\leq j\leq n$ and $i\neq j$. By this, the result follows.
\end{proof}

The lemma below contains an easy property of the reduction.

\begin{lemma}
Let $M$, $U$, $N$ be given such that $M\ra_{\cal R} U$. Then $M[x:=N]\ra_{\cal R} U[x:=N]$.
\end{lemma}

\begin{proof}
By induction on $M$.\\
\end{proof}

We reveal two observations in the next lemma before verifying the condition {(C5)}.

\begin{lemma}\label{propagaSN}~
\begin{enumerate}
\item Let $M[x:=N] \in {\cal  SN}_{\cal R}$ for some terms $M$, $N$ and variable $x\in {\cal V}_{\l}$. Then $M \in  {\cal  SN}_{\cal R}$.
\item Let $M[\a:=N] \in {\cal  SN}_{\cal R}$ for some terms $M$, $N$ and variable $\a\in {\cal V}_{\m}$. Then $M \in  {\cal  SN}_{\cal R}$.
\end{enumerate}
\end{lemma}

\begin{proof}
We deal only with Point 1, the other case being similar. Let $M\ra_{\cal R} U$. We prove that $U\in {\cal SN}_{\cal R}$. We argue by induction on $\eta(M[x:=N])$. From the assumption $M\ra_{\cal R} U$, we deduce $M[x:=N]\ra_{\cal R} U[x:=N]$ by the previous lemma. Since $M[x:=N]\in {\cal SN}_{\cal R}$, we have $U[x:=N]\in {\cal SN}_{\cal R}$ and $\eta(U[x:=N]) < \eta(M[x:=N])$, which, by IH, gives the result. 
\end{proof}

\begin{lemma}[Condition (C5)] \label{C5sn}
Let $(M[x:=N])\bar{P} \in {\B}$, $N \in {\B}$ and \\$(\l x.M)N\bar{P}  \in  {\cal T}_t$. Then $(\l x.M)N\bar{P}  \in  {\B}$.
\end{lemma}

\begin{proof}
Firstly, we observe that $(M[x:=N])\bar{P} \in {\cal SN}_{\cal R}$ implies $M$, $\bar{P} \in {\cal SN}_{\cal R}$. We proceed by induction on $\eta(M)+\eta(N)+\eta(\bar{P})$ 
where $\eta(\bar{P})$ is the sum of the $\eta$s of the terms of the sequence $\bar{P}$. Let   $(\l x.M)N\bar{P} \ra_{\cal R} Q$. We examine two cases.
\begin{itemize}
\item If $Q = (\l x.M')N'\bar{P'} $ with $M \ra_{\cal R} M'$ or $N \ra_{\cal R} N'$ or $\bar{P}  \ra_{\cal R} \bar{P'}$, then we can apply IH.
\item If $Q = (M[x:=N])\bar{P}$, we have the result by assumption.
\end{itemize} 
\end{proof}

\begin{lemma}[Condition (C6)] \label{C6sn}
Let $\mu \a.M[\a:=\bar{N}] \in \B$, $\bar{N} \in {\B}^{<\omega}$ and \\$(\mu \a.M)\bar{N} \in  {\cal T}_t$. Then $(\mu \a.M)\bar{N} \in  {\B}$.
\end{lemma}

\begin{proof}
We observe that, by Lemma \ref{propagaSN}, $\mu \a.M[\a:=\bar{N}] \in {\cal SN}_{\cal R}$ implies $M \in {\cal SN}_{\cal R}$. 
We note that, if $\bar{N} = N_1N_2\dots N_n$, then $M[\a:=\bar{N}] = M[\a:=N_1][\a:=N_2] \dots [\a:=N_n]$.
Therefore, we need to prove that, if $(\mu \a.M[\a:=N]) \bar{P} \in {\cal SN}_{\cal R}$ with $N \in {\cal SN}_{\cal R}$ and $\bar{P}\in {\cal SN}_{\cal R}^{<\omega}$,  then $(\mu \a.M)N\bar{P} \in  {\cal SN}_{\cal R}$.
We proceed by induction on $\eta(M)+\eta(N)+\eta(\bar{P})$ 
where $\eta(\bar{P})$ is the sum of the $\eta$s of the terms of the sequence $\bar{P}$. Let   $(\mu \a.M)N\bar{P}  \ra_{\cal R} Q$. We examine two cases.
\begin{itemize}
\item If $Q = (\mu \a.M')N'\bar{P'} $ with $M \ra_{\cal R} M'$  or $N \ra_{\cal R} N'$ or $\bar{P}  \ra_{\cal R} \bar{P'}$, then we can apply IH.
\item If $Q = (\mu \a. M[\a:=N])\bar{P}$, we have the result.
\end{itemize} 
\end{proof}

We are now in a position to state the main result of this section.

\begin{theorem}\label{SN}
$\B$ is a saturated set.
\end{theorem}

\begin{proof}
We put together Lemmas \ref{C1sn}, \ref{C2sn}, \ref{C3sn}, \ref{C4sn}, \ref{C5sn} and \ref{C6sn}.
\end{proof}

\begin{corollary}\label{sncorollary}
If $M  \in {\cal T}_t$, then $M \in  {\cal SN}_{\b\m\r\th\e}$.
\end{corollary}
   
\begin{proof}
We apply Corollary \ref{closed} and Theorem \ref{SN}.
\end{proof}

\section{Weak normalization property of $\b\m\m'\r\e\th$-reduction}

In this section, we prove that $\b\m\r\e\th$-reduction augmented with the $\m'$-rule has the weak normalization property. In \cite{BN3}, we provided a syntactic proof of this result. However, that proof was notably complex, involving the strong normalization of $\beta$-reduction within a typed framework and the weak normalization of the other reductions within an untyped framework. The most challenging part was devising an algorithm to combine the two normalization results to establish the weak normalization of the set of all reductions. The proof presented here takes a different approach. It relies on the results of Section 3 by choosing a suitable set that satisfies all the properties required for the saturation. We observe that condition (C4) is the most difficult one to establish, and, interestingly, it will be formulated and verified within the untyped context.  While this proof for the weak normalization property is elegant and concise, it lacks the constructive aspect found in the previous proof \cite{BN3}, and as a result, we are unable to extract a concrete normalization algorithm.

Additionally, it is worth noting that, except for verifying condition (C2) for $\B'$, none of the other properties require typability assumptions. While the formulations of the lemmas mention typability assumptions to align with Definition \ref{def:RBotgood}, we will present proofs that do not rely on types for the conditions with the exception of (C2).

We first show in the next two lemmas that, for the proof of the weak normalization property of $\b\m\m'\r\e\th$-reduction, it is sufficient to verify that $\b\m\m'\r\e$-reduction enjoys weak normalization. This is because we can add $\theta$-reduction to the set of reductions without compromising the weak normalization property.

\begin{lemma} \label{theta1}
Let us suppose $M \in {\cal NF}_{\b\m\m'\r\e}$ and $M \two_{\th} N$ for some $M$, $N\in \mathcal{T}$. Then $N \in {\cal NF}_{\b\m\m'\r\e}$ and, if $N$ starts with $\l$ (with $\m$, resp.), then  $M$  also starts with $\l$ or $\m$ (with $\m$, resp.).
\end{lemma}

\begin{proof}
By induction on $M$. It suffices to check the property for one-step of reduction.
\begin{itemize}
\item If $M = \l x.M'$, the result is trivial.
\item If $M = \m \a.M'$ and $N =  \mu \a. N'$ where $M' \ra_{\th} N'$, then, by IH, 
$N' \in {\cal NF}_{\b\m\m'\r\e}$ and $N'$ does not start with $\m$. Thus  $N \in {\cal NF}_{\b\m\m'\r\e}$ and the second property is obviously verified.
\item If $M = \m \a.[\a]M'$, $\a \not\in {\rm fv}(M')$ and $N =  M'$, then 
$M' \in {\cal NF}_{\b\m\m'\r\e}$ and the second property is obviously verified.
\item If $M = [\a]M'$, then $N =  [\a] N'$ where $M' \ra_{\th} N'$. Then, by IH, 
$N' \in {\cal NF}_{\b\m\m'\r\e}$ and $N'$ does not start with $\m$. Hence,  $N \in {\cal NF}_{\b\m\m'\r\e}$.
\item If $M = (M_1)M_2$, $N=(M'_1)M_2$ and  $M_1 \ra_{\th} M'_1$, then $M'_1,M_2 \in {\cal NF}_{\b\m\m'\r\e}$ and $M'_1$ does not start with $\m$ or $\l$. Hence,  $N \in {\cal NF}_{\b\m\m'\r\e}$.
\item If $M = (M_1)M_2$, $N=(M_1)M'_2$ and  $M_2 \ra_{\th} M'_2$, then $M_1,M'_2 \in {\cal NF}_{\b\m\m'\r\e}$ and $M'_2$ 
does not start with $\m$. Hence, $N \in {\cal NF}_{\b\m\m'\r\e}$.
\end{itemize}
\end{proof}

\begin{lemma}\label{theta2}
The $\theta$-reduction strongly normalizes.	
\end{lemma}

\begin{proof}
We observe that $\theta$-reduction decreases the size of the terms.
\end{proof}

\begin{theorem}\label{conc}
${\cal WN}_{\b \m \m' \rho \e} \subseteq {\cal WN}_{\b \m \m' \rho \e \th}$ holds true.
\end{theorem}

\begin{proof} 
If $M \in {\cal WN}_{\b \m \m' \rho \e}$, then $\exists$ $M'$ such that $M \two_{\b \m \m' \rho \e} M'$ and $M' \in  {\cal NF}_{\b \m \m' \rho \e}$. Thus, by Lemmas \ref{theta1} and \ref{theta2}, $\exists$ $N \in {\cal NF}_{\b \m \m' \rho \e \th}$ such that $M' \two_{\th}N$.\\
\end{proof}

Let us now recall the example of our paper \cite{BN3} which shows that with the reduction $\m'$ we lose the strong normalization property.
Let $M=(\m \b.U)U$ where $U  =\m \a. [\a][\a]x$, then $x:\bot \v M : \bot$ and there are $M_1$, $M_2$, $M_3$ such that 
$M \ra_{\m'} M_1 \ra_{\m} M_2 \ra_{\r} M_3 \ra_{\th} M$, which means $M \not \in {\cal SN}_{\b\m\m'\r\e\th}$. Hence, we cannot hope to prove that $\b\m\m'\r\e$-reduction strongly normalizes. Instead, we intend to show that $\b\m\m'\r\e$-reduction enjoys the weak normalization property.

In what follows, we let ${{\cal R}'} = \{\b,\m,\m',\r,\e\} $ and $\B' = {\cal WN}_{\cal{R}'} \cap {\cal T}_t$. Our aim is to verify that $\B'$ is saturated, from which, by Corollary \ref{closed}, it follows that every typable term is weakly normalizable. 
We verify the conditions of Definition \ref{def:RBotgood} one by one below. Beforehand, we deal with some lemmas that will be needed in the proofs. Lemma \ref{alphabetawn} is similar to Lemma \ref{alphabetasn}.

\begin{lemma} Let $M\in {\cal T}$ and $\a,\b\in {\cal V}_{\m}$. Then $M\ra_{{\cal R}'}  N$ iff $M[\a:=\b]\ra_{{\cal R}'} N[\a:=\b]$. 
\end{lemma} 

\begin{proof}
By induction on $M$.
\end{proof}

\begin{lemma}\label{alphabetawn}
Let  $M\in {\cal T}$ and $\a,\b\in {\cal V}_{\m}$. Then $M\in {\cal NF}_{{\cal R}'}$ iff $M[\a:=\b]\in {\cal NF}_{{\cal R}'}$.
\end{lemma}

\begin{proof}
Follows from the previous lemma.
\end{proof}


\begin{lemma}\label{Lalpha}
Let $M \in {\cal T}_t$  and $\a \in {\cal V}_{\m}$.
If $M \in \mathcal{NF}_{{\cal R}'}$, then $M_{\a}\in \mathcal{NF}_{{\cal R}'}$. Moreover, the following statements hold.
\begin{enumerate}
\item If $M_\a$ starts with $\m$, then $M$ starts with $\m$.
\item If $M_\a$ starts with $\l$, then $M$ starts with $\l$ or $[.]$.
\end{enumerate}
\end{lemma}

\begin{proof}
By induction on $M$. We detail only the more interesting cases.
\begin{itemize}
\item If $M = [\b]M'$ and $\b\neq \a$, then $M'$  does not start with $\m$,  $M_{\a} = [\b]M'_{\a}$ and, by IH, $M'_{\a}\in \mathcal{NF}_{{\cal R}'}$ and $M'_{\a}$ does not start with $\m$. Hence, we have the result.
\item If $M = [\a]M'$, then $M'$  does not start with $\m$,  and $M_{\a} = M'_{\a}$. By IH, $M'_{\a}\in \mathcal{NF}_{{\cal R}'}$ and $M'_{\a}$ does not start with $\m$. Obviously, if $M_{\a}$ starts with $\l$, then $M$ starts with $[.]$.
\item If $M = (P)Q$, then $P, Q$ do not start with $\m$, $P$ does not start with $\l$ and $M_{\a} = (P_{\a})Q_{\a}$. By IH, $P_{\a}, Q_{\a}\in {\cal NF}_{{\cal R}'}$ and 
$P_\a, Q_\a$ do not start with $\m$. If $P_\a$ starts with $\l$, then $P$ starts either with $\l$ or with $[.]$. Since $M\in {\cal NF}_{{\cal R}'}$, $P$ cannot start with  $\l$ and, since $M\in {\cal T}_t$, $P$ cannot start with $[.]$, either. Hence, $M_\a\in {\cal NF}_{{\cal R}'}$  
\end{itemize}
\end{proof}

\begin{remark}
In the lemma above we obviously need the assumption that $M$ is typable. Otherwise, let $M=([\a]\l x.P)\,Q$. Then $M\in \mathcal{NF}_{{\cal R}'}$, but $M_\a=(\l x.P)\,Q\notin\mathcal{NF}_{{\cal R}'}$.
\end{remark}

The next lemma is intuitive and has the consequence that we are not able to create a $\m$, $\l$ or $[.]$ by a $\m$- or $\m'$-substitution or by replacing a $\m$-variable with another one.

\begin{lemma}\label{bl[]}
 Let $M,N \in {\cal T}$, $\a \in {\cal V}_{\m}$ and $s \in \{r,l\}$. 
 \begin{enumerate}
 \item If $M[\a :=_s N]$ starts with $\l$ (resp. $\m$, $[.]$), then $M$ also  starts with $\l$ (resp. $\m$, $[.]$).
 \item If $M[\a :=\b]$ starts with $\l$ (resp. $\m$, $[.]$), then $M$ also  starts with $\l$ (resp. $\m$, $[.]$). 
 \end{enumerate}
\end{lemma}

\begin{proof}
By induction on $M$.
\end{proof}

\begin{lemma}[Condition (C1)]\label{C1wn}
Let us suppose $M \in \B'$. Then $\l x.M \in \B'$.
\end{lemma}

\begin{proof}
Obvious.
\end{proof}

\begin{lemma}[Condition (C2)] \label{C2wn}
Let $M \in \B'$ and $\a \in {\cal V}_{\m}$.
If $\m \a.M \in {\cal T}_t$, then $\m \a.M \in {\B'}$.
\end{lemma}

\begin{proof}
We consider a normalization of $M$, i.e., $M \two N\in {\cal NF}_{\cal{R}'}$. We take the reduction sequence $\m \a.M \two \m \a.N$ obtained thereof. We distinguish the different cases.
\begin{itemize}
\item If $N$ does not start with $\m$, then $\m \a. N \in {\cal NF}_{\cal{R}'}$ (note that $\th\notin{\cal R}'$), thus  $\m \a.M \in {\cal WN}_{\cal{R}'}$.
\item If $N =  \m \b. P$, then $P$ does not start with $\m$. 
By Lemma \ref{Lalpha}, we have $\m \a.M \two \m \a. \m \b.P \ra_{\e} \m \a. P_{\b}\in {\cal NF}_{\cal{R}'}$. Hence,  $\m \a.M \in {\cal WN}_{\cal{R}'}$.
\end{itemize}
\end{proof}

We demonstrate the fact that the assumption $\m \a.M \in {\cal T}_t$ is needed in the lemma above by giving a simple example.
Let $M=\m\b.([\b]\l y.(y)\delta)\delta$, where $\delta=\l x.(x)x$. Then $M\in {\cal NF}_{\cal{R}'}$ and $\m\a.M\ra_\varepsilon \m\a.M_\b=\m\a. (\l y.(y)\delta)\delta\ra_\b\m\a.(\delta)\delta$. Hence, $\m\a.M\notin {\cal WN}_{\cal{R}'}$.

\begin{lemma}[Condition (C3)] \label{C3wn}
Let $M \in \B'$ and $\a \in {\cal V}_{\m}$.
If $[\a]M \in {\cal T}_t$, then $[\a]M \in {\B'}$.
\end{lemma}

\begin{proof}
We consider a normalization of $M$, i.e., $M \two N\in {\cal NF}_{\cal{R}'}$. Then we take the reduction sequence $[\a]M \two[\a] N$ obtained thereof.
We distinguish the different cases.
\begin{itemize}
\item If $N$ does not start with $\m$, then $[\a] N \in {\cal NF}_{\cal{R}'}$. Thus  $[\a]M \in {\cal WN}_{\cal{R}'}$.
\item If $N =  \m \b. P$, then, by Lemma \ref{alphabetawn}, we have $[\a]M \two [\a]\m \b.P \ra_{\r}  P[\b:=\a]  \in {\cal NF}_{\cal{R}'}$, then $[\a]M \in {\cal WN}_{\cal{R}'}$.
\end{itemize}
\end{proof}

Next, we are going to prove the condition {\bf (C4)} for $\B'$. 

\begin{lemma}[Condition (C4)] \label{C4wn}
Let $n\geq 0$, $ N_1, \dots, N_n \in {\B'}$ and $x \in {\cal V}_{\l}$, then $(x)N_1 \dots N_n \in {\B'}$.
\end{lemma}

We prove the lemma with the help of several auxiliary lemmas. 
Our intuition is that, when we are given a term $(x)N_1 \dots N_n$ with $N_1, \dots, N_n \in {\cal WN}$, 
then we move from left to right until we find the first $1\leq i\leq n$  such that $N_i=\m\g.N_i'$ for some $N_i'$. 
We normalize the term $(x)N_1 \dots N_i$, and, after obtaining the normal form $\m\g.N_i''$, we proceed
with normalizing $(\m\g.N_i'')N_{i+1} \dots N_n$. The normalization strategy is a little tricky at that point, however. When the next term $N_{i+1}$ does not start with a $\m$, then we perform a $\mu$-reduction for the redex  $(\m\g.N_i'')N_{i+1}$. On the other hand, if $N_{i+1}=\m\g_{i+1}.N_{i+1}'$, then  we continue with a $\m'$-reduction concerning the $\m'$-redex $(\m\g.N_i'')\m\g_{i+1}.N_{i+1}'$. We continue in this way for the remaining components of the application.
First of all, we introduce some necessary notions.\\

\begin{definition}[$\a$-clean property] Let $M \in {\cal T}_t$ and $\a \in {\cal V}_{\m}$. We say that $M$ is $\a$-clean if, 
for every subterm $[\a]U$ of $M$, $U$ does not start with $\l$.
\end{definition}

Intuitively, $\a$-clean terms do not create new $\b$-redexes when a $\m$-substitution is considered with respect to $\a$. 

To establish condition (C4), we rely on some technical lemmas (Lemmas \ref{WN2lem1}, \ref{WN2lem2}, and \ref{WN2lem3}) to attain a weak normalization result. As we have seen in the above explanation, we need to normalize the result of a $\mu$-substitution of a normal $\lambda\mu$-term into another normal $\lambda\mu$-term. The difficulty in such normalization arises from the possibility of encountering $\beta$-redexes. The $\alpha$-clean condition aims to prevent these occurrences. We achieve this normalization result in two steps. First, we normalize without the $\mu'$-reduction (without the $\mu$-reduction, resp.) and we describe precisely the remaining $\mu'$-redexes ($\mu$-redexes, resp.). Then, we handle the final normalization by eliminating these redexes.

\begin{lemma}\label{WN2lem1}\hfill
\begin{enumerate}
\item Let $M$, $N\in {\cal NF}_{\cal R'}$ such that $M$ is $\a$-clean. Then $M[\a:=_rN]\in {\cal NF}_{\b\m\r\e}$ is $\a$-clean and the $\m'$-redexes of $M[\a:=_rN]$ are of the form $[\a](U)N$ if $N=\m \b.N'$. In particular, if $N\neq \m \b.N'$, then $M[\a:=_rN]\in {\cal NF_{\cal R'}}$ as well.
\item Let $M$, $N\in {\cal NF_{\cal R'}}$ such that $N \neq \l x.N'$ for some $N'$. Then $M[\a:=_lN]\in {\cal NF}_{\b\m'\r\e}$ and the $\m$-redexes of $M[\a:=_lN]$ are of the form $[\a](N)U$ if $N=\m \b.N'$. In particular, if $N\neq \m \b.N'$, then $M[\a:=_lN]\in {\cal NF_{\cal R'}}$ as well.
\end{enumerate}
\end{lemma}

\begin{proof}
Both points can be proved by induction on $M$. 
\begin{enumerate}
\item
\begin{itemize}
\item In case of $M=\l x.M'$,  we apply the induction hypothesis.
\item If $M=\m\b.M'$, then $M'$ does not start with $\m$ and $M[\a:=_rN]=\m\b.M'[\a:=_rN]$.
 By Lemma \ref{bl[]}, $M'[\a:=_rN]$ does not start with $\m$ and we apply IH on $M'$ to obtain the result.
\item If $M=(M_1)M_2$, then  $M_1$ does not start with $\m$ or $\l$, $M_2$ does not start with $\m$ and 
$M[\a:=_rN]=(M_1[\a:=_rN])M_2[\a:=_rN]$. By Lemma \ref{bl[]}, $M_1[\a:=_rN]$ does not start with $\m$ or $\l$,
$M_2[\a:=_rN]$ does not start with $\m$ and, applying IH on $M_1$, $M_2$, we obtain the result.
\item If $M=[\b]M'$ and $\b\neq \a$, then $M'$ does not start with $\m$  and  $M[\a:=_rN]=[\b]M'[\a:=_rN]$. 
By Lemma \ref{bl[]}, $M'[\a:=_rN]$ does not start with $\m$ and the result follows from IH. 
\item If  $M = [\a]M'$, then $M'$ does not start with $\m$  and $M[\a:=_rN]=[\a](M'[\a:=_rN])N$.
Since $M$ is $\a$-clean, $M'$ does not start with $\l$ and,  by Lemma \ref{bl[]}, $M'[\a:=_rN]$ does not start with $\m$ or $\l$. 
By applying IH, we see that $[\a](M'[\a:=_rN])N \in {\cal NF}_{\b\m\r\e}$ is $\a$-clean and the $\m'$-redexes are of the desired form.
\end{itemize}
\item
\begin{itemize}
\item In case of $M=\l x.M'$, $M=\m\b.M'$, $ M=(M_1)M_2$ and $M=[\b]M'$ ($\a\neq \b$), we apply, as in Point 1, IH.
\item If  $M = [\a]M'$, then $M'$ does not start with $\m$  and $M[\a:=_lN]=[\a](N)M'[\a:=_lN]$.
By Lemma \ref{bl[]}, $M'[\a:=_rN]$ does not start with $\m$. 
By virtue of IH and since $N$ does not start with $\l$, we see that $[\a](N)M_1[\a:=_lN] \in {\cal NF}_{\b\m'\r\e}$ and the $\m$-redexes are of the desired form.
\end{itemize}
\end{enumerate}
\end{proof}

\begin{lemma}\label{WN2lem2}
Let  $M, N \in  {\cal NF}_{\cal R'}$ such that $M\neq \l x.M'$ and, if $M=\m\a.M'$, then $M$ is $\a$-clean. Then $N[\g:=_lM]\two_{\m\r}P\in {\cal NF}_{\cal R'}$ and $P$ is $\g$-clean. 
Moreover, if $N$ does not start with $\m$, then the same is true for $P$.  
\end{lemma}

\begin{proof}
The proof proceeds by induction on $N$.
\begin{itemize}
\item If $N=\l x.N'$, then the assertion is straightforward.
\item If $N=[\b]N'$ ($\b\neq \g$). Then $N'$ does not start with $\m$. By IH, $N'[\g:=_lM]\two_{\m\r}P'\in {\cal NF}_{\cal R'}$ and $P'$ does not start with $\m$. Hence, $[\b]P'\in {\cal NF}_{\cal R'}$.
\item If $N=[\g]N'$, then $N[\g:=_lM]=[\g](M)N'[\g:=_lM]$ and $N'$ does not start with $\m$.  
Hence by Lemma \ref{bl[]}, $N'[\g:=_lM]$ does not start with $\m$ either. By IH, $N'[\g:=_lM]\two_{\m\r}P'\in {\cal NF}_{\cal R'}$ and $P'$ does not start with $\m$. If $M$ does not start with $\m$, then we are ready. Otherwise, assume $M=\m\a.M'$. Then $M'$ does not start with $\m$ and $M$ is $\a$-clean. In this case, $N[\g:=_lM]=[\g](\m\a.M')N'[\g:=_lM]\ra_{\m}[\g]\m\a.M'[\a:=_rN'[\g:=_lM]]\ra_\r M'[\a:=_rN'[\g:=_lM]][\a:=\g]\two_{\m\r}P\in {\cal NF}_{\cal R'}$ by applying Lemmas \ref{WN2lem1} and \ref{alphabetawn} and IH. From Lemma \ref{bl[]}, it follows that $P$ does not start with $\m$ and, by IH, $P$ is $\g$-clean.   
\item $N=\m\b.N'$. Then we can apply IH.
\item $N=(N_1)N_2$. Then  $N_1$ does not start with $\l$ or $\m$ and $N_2$ does not start with $\m$.  By using IH, we obtain the result. 
\end{itemize} 
\end{proof}

\begin{lemma}\label{WN2lem3}
Let $P, Q\in {\cal NF_{\cal R'}}$ and $\a \in {\cal V}_{\m}$. Then the following statements hold.
\begin{enumerate}
\item Let $P=\m\a.P'$ such that $P$ is $\a$-clean. Assume $Q$ does not start with $\m$. Then $\exists$ $R$ for which $(P)Q\ra_\m \m\a.R\in {\cal NF}_{\cal R'}$ and $R$ is $\a$-clean.
\item Let $Q=\m\g.Q'$. Assume $P\neq \l x.P'$ and, if $P=\m\a.P'$, then $P$ is $\a$-clean. Then $\exists$ $R$ for which $(P)Q\two_{\m\m'\r} \m\g.R\in {\cal NF}_{\cal R'}$ and $R$ is $\g$-clean.
\end{enumerate}
\end{lemma}

\begin{proof}
Let $P$, $Q$ be as in the lemma. We verify the two statements of the lemma.
\begin{enumerate}
\item Assume $P=\m\a.P'$ such that $P$ is $\a$-clean and $Q$ does not start with $\m$. Then we apply Lemma \ref{WN2lem1} to $(\m\a.P')Q\ra_{\m}\m\a.P'[\a:=_rQ]$ to obtain the result.
\item We apply the previous lemma to $(P)Q\ra_{\m'}\m\g.Q'[\g:=_lP]$.
\end{enumerate}
\end{proof}


\vspace{10pt}
\noindent{\bf Proof of Lemma \ref{C4wn} } Let us assume we are given terms $N_1, \dots, N_n \in {\cal WN}_{\cal R'}$, where $x \in {\cal V}_{\l}$. We shall assume $n>0$, as, otherwise, the statement is trivial. 
Let $1\leq i\leq n$ be the first index such that $N_i=\m\g.N_1'$. We apply $\m'$-reductions to $(x)N_1\ldots N_{i-1}\,\m\g_iN_i'$ to obtain $\m\g_i.N_i'[\g_i:=_l(x)N_1\ldots N_{i-1}]=\m\g_i.N'$. Then, according to the lemma above, $\m\g_i.N'$ is $\g_i$-clean and, by Lemma \ref{WN2lem1}, $\m\g_i.N'\in {\cal NF}_{\cal R'}$. The remaining term to normalize is $(\m\g_i.N')N_{i+1}\ldots N_n$. We prove by induction on $i$ that $(\m\g_i.N')N_{i+1}\ldots N_n\in {\cal WN}_{\cal R'}$. Assume $N_{i+1}$ does not start with $\m$. Then Point 1 of Lemma \ref{WN2lem3} applies, and we obtain $(\m\g_i.N')N_{i+1}\ra_\m \m\g_i.N''$ such that $\m\g_i.N''$ is $\g_i$-clean and $\m\g_i.N''\in {\cal NF}_{\cal R'}$. By IH, $(\m\g_i.N'')N_{i+2}\ldots N_n\in {\cal WN}_{\cal R'}$. On the other hand, suppose $N_{i+1}=\m\g_{i+1}.N_{i+1}'$. Applying Point 2 of Lemma \ref{WN2lem3} to $(\m\g_i.N')\m\g_{i+1}.N_{i+1}'$, we obtain $\m\g_{i+1}.R\in {\cal  NF}_{\cal R'}$ such that $R$ is $\g_{i+1}$-clean and we conclude that $(\m\g_{i+1}.R)N_{i+2}\ldots N_n\in {\cal WN}_{\cal R'}$.\phantom{aaa}\hfill$\square$\\

We can turn to the conditions (C5) and (C6).

\begin{lemma}[Condition (C5)]\label{C5wn}
Let $M,N \in {{\cal T}_t}$, $N\in \B'$ and $\bar{P} \in {{\cal T}_t}^{<\omega}$.
If  $(M[x:=N])\bar{P} \in \B'$ and $(\l x.M)N\bar{P}  \in {\cal T}_t$ , then $(\l x.M)N\bar{P}  \in  {\B'}$.
\end{lemma}

\begin{proof}
Indeed, we have $(\l x.M)N\bar{P} \ra_{\b}  (M[x:=N])\bar{P} \in {\cal WN}_{\cal R'}$.
\end{proof}

\begin{lemma}[Condition (C6)]\label{C6wn}
Let $M,N \in {\cal T}_t$ and $\bar{N} \in ({\cal{\B'}})^{<\omega}$.
If $\mu \a.M[\a:=_r \bar{N}] \in \B'$ and $(\mu \a.M)\bar{N} \in  {\cal T}_t$, then  $(\mu \a.M)\bar{N}  \in  {\B'}$.
\end{lemma}

\begin{proof}
Indeed, we have $(\mu \a.M)\bar{N}  \two_{\m} \mu \a.M[\a:=_r\bar{N} ] \in {\cal WN}_{\cal R'}$.
\end{proof}

We are now in a position to state and prove the main theorem of this section. 

\begin{theorem}\label{B'}
$\B'$ is a saturated set.
\end{theorem}

\begin{proof}
We use Lemmas \ref{C1wn}, \ref{C2wn}, \ref{C3wn}, \ref{C4wn}, \ref{C5wn} and \ref{C6wn}.
\end{proof}

\begin{theorem}\label{WN}
If $M \in {\cal T}_t$, then $M \in  \mathcal{WN}_{ \b\m\m'\r\e\th}$.
\end{theorem}

\begin{proof}
We apply Corollary \ref{closed} and Theorems \ref{conc} and \ref{B'}.
\end{proof}

\begin{remark}~\rm 
As a final remark, we would emphasize that the weak normalization property follows for the typable terms without any additional effort. Notably, verifying that the set ${\cal WN}_{\b\m\r\e\th} \cap {\cal T}_t$ is saturated is easy, as condition (C3) becomes trivial with the absence of the $\m'$-rule. Thus, if $M$ is a typable $\l\m$-term, then $M\in{\cal WN}_{\b\m\r\e\th}$. 
\end{remark}

\section{Future work}

\begin{itemize}

\item We have proven normalization results (both strong and weak) for sets of reduction rules in the simply typed $\lambda\mu$-calculus. We believe that this method is sufficiently general to allow us to establish analogous results when we introduce additional reductions. For instance, we can consider the following new typing rule:
$$[\a][\b]M \ra_{\delta_1} [\b]M_{\a}\;\;\;{\rm if} \;\b\neq \a \;\;\;\;\;\;\;\; {\rm and}\;\;\;\;\;\;\;\;  [\a][\a]M \ra_{\delta_2} M_{\a}.$$
These rules can be justified both on the basis of typing considerations and the aim of simplifying terms by eliminating unnecessary parts of a term. Notably, the presence of two consecutive brackets $[.]$ is one of the factors contributing to the failure of strong normalization, leading only to weak normalization.
Given the declarations $\alpha : \bot$ and $\beta : A$, and assuming $M : A$, we can deduce $[\alpha][\beta]M : \bot$.
This derivation implicitly contains a proof of $\perp$ by contradiction.  However, this intermediate step is clearly redundant: it would be more direct to derive a contradiction from the assumptions $M:A$ and $\alpha: A$, the latter representing a proof of $\neg A$. 
In particular, where $\beta = \alpha$, both brackets can be entirely eliminated. We can investigate the normalization properties (both weak and strong) of a set of rules that includes this new reduction rule $\delta$. 
Other non-local rules can also be added to the $\lambda\mu$-calculus. In \cite{Nou}, a global reduction rule made it possible, for instance, to encode a ``parallel or''. We believe that the method used in that paper to prove normalization properties makes it immediately clear what conditions are required to obtain such a result.

\item We have already mentioned that the reduction rule $\m'$ does not preserve types during reduction in a typed system based on second-order logic, and, that Raffalli provided an example of a term of type $A$ that can be reduced (using the rule $\mu'$) to a term of type $B$, where $B$ can be any type. This result appeared in Py's thesis \cite{Py}. To address this problem, the main idea is to give algorithmic content to the rules associated with the quantifier $\forall$. Py has already accomplished this in the previously cited work. This extended version of the calculus, combining higher-order types, new term-level constructors, and novel reduction rules, provides a rich and promising framework for both theoretical investigation and practical application. However, the study of such extensions presents several significant challenges. In particular, it is essential to establish foundational properties such as type preservation under reduction, the expressivity and correctness of programming over algebraic data types, including their faithful representation within the calculus, and the effective extraction of the algorithmic content of mathematical proofs. These goals require a precise understanding of the interplay between typing and computation, making the analysis of this system far from trivial. Addressing these issues calls for advanced techniques in type theory, semantics, and normalization theory.
It is known that the arithmetical proofs developed in previous works do not extend to this enriched system. It is therefore of particular interest to investigate to what extent the method of reducibility candidates, as presented in this paper, can be adopted in order to establish weak and/or strong normalization results for this expressive framework. Such an extension would represent a significant advancement, given the complexity introduced by the richness of the system.

\item Some of the saturation conditions we have defined, as well as certain additional hypotheses, were motivated by our aim to prove normalization for certain sets of reduction in the $\lambda\mu$-calculus. Therefore, simpler conditions might also suffice when we want to obtain other type of results, for example, a result of adequacy, particularly for the study of the algorithmic behavior of certain typable $\lambda\mu$-terms via this semantics. In \cite{Nou3}, we managed to define, for a fragment of the $\lambda\mu$-calculus, a realizability semantics that is, via a completeness result, established an equivalence between terms typable by a type and terms inhabiting all interpretations of the type. While this result is interesting, it relies on the introduction of several kinds of $\B$, whose interpretation is not entirely straightforward. It would be of interest to investigate how these results carry over when using our current semantic framework.

\item In many typed calculi based on classical logic, call-by-name and call-by-value strategies naturally emerge and can often be clearly defined. However, this is less evident in the $\lambda\mu$-calculus, where defining the notion of value and a corresponding call-by-value strategy is neither straightforward nor intuitive. Py \cite{Py} introduced such a system, in which the $\mu'$ rule plays an essential role, and established confluence by eliminating certain non-confluent critical pairs. He also proved a normalization result, though limited to a very small fragment of the calculus. It would be interesting to investigate a full call-by-value variant of the $\lambda\mu$-calculus and explore whether the techniques developed in the present paper could be adapted to advance this line of inquiry. Preliminary observations suggest that this task is quite delicate, as it may require revisiting some of the core saturation criteria.

\end{itemize}

\section*{Acknowledgments}
We would like to express our sincere gratitude to Miquel for his invaluable assistance in reformulating several definitions and highlighting one of the main results of this paper, namely the characterization of typed terms. We also extend our thanks to the three referees, whose insightful comments and suggestions significantly improved the paper, particularly regarding its presentation, potential connections to other works in the field, and directions for future research.

\end{document}